%% file: ZHU_FP_final.tex
\documentclass[11pt]{amsart}
\usepackage[margin=3cm]{geometry}
\usepackage{mathrsfs}
\usepackage{amsmath,amsfonts,amssymb,amsthm}
\usepackage{bbm}
\usepackage{epstopdf}
\usepackage{graphicx}
\usepackage[linktocpage=true,colorlinks=true,linkcolor=blue,citecolor=red,urlcolor=magenta,pdfborder={0 0 0}]{hyperref}
\usepackage{fancyhdr}
\theoremstyle{plain}
\newtheorem{theorem}{Theorem}[section]
\newtheorem{corollary}[theorem]{Corollary}
\newtheorem{lemma}[theorem]{Lemma}
\theoremstyle{definition}
\newtheorem{remark}[theorem]{Remark}
\newtheorem*{notation}{Notation}
\numberwithin{equation}{section}

\newcommand*{\dif}{\mathop{}\!\mathrm{d}}

\begin{document}
	\title{Velocity averaging and Hölder regularity for kinetic Fokker-Planck equations with general transport operators and rough coefficients}
	\author{Yuzhe ZHU}
	\date{\today}
	\address{D\'epartement de math\'ematiques et applications, \'Ecole normale sup\'erieure, Paris, France}
	\email{yuzhe.zhu@ens.fr}

\begin{abstract}
This article addresses the local boundedness and Hölder continuity of weak solutions to kinetic Fokker-Planck equations with general transport operators and rough coefficients. These results are due to the mixing effect of diffusion and transport. Although the equation is parabolic only in the velocity variable, it has a hypoelliptic structure provided that the transport part $\partial_t+b(v)\cdot\nabla_x$ is nondegenerate in some sense. We achieve the results by revisiting the method, proposed by Golse, Imbert, Mouhot and Vasseur in the case $b(v)= v$, that combines the elliptic De~Giorgi-Nash-Moser theory with velocity averaging lemmas. 
\end{abstract}
\maketitle

\section{Introduction}
We are concerned with the Fokker-Planck equation of the form
\begin{equation}\label{FP}
(\partial_t+b(v)\cdot\nabla_x)f={\rm div}_v\left(A(t,x,v)\nabla_v f\right)+B(t,x,v)\cdot\nabla_vf+s(t,x,v)
\end{equation}
in $(t,x,v)\in[t_1,t_2]\times U_x\times U_v$, for $-\infty<t_1<t_2\le\infty$ and some open sets $U_x\subset\mathbb{R}^{d_1}$, $U_v\subset\mathbb{R}^{d_2}$, where the vector fields $b$ and $B$ are valued in $\mathbb{R}^{d_1}$ and $\mathbb{R}^{d_2}$, respectively. 
In the whole article, we suppose that $b$ only depends on the velocity variable $v$. We also make the assumption
\begin{equation} \label{H}
 \left\{ 
 \begin{aligned}
 \ &0<\lambda I\le A\le\Lambda I, \\
 \ &\| B\|_{L^\infty}\le\Lambda, \\
 \end{aligned}
 \right. 
 \end{equation} 
 for some constants $\lambda$, $\Lambda>0$, where $A$ denotes a ${d_2}\times{d_2}$ real symmetric matrix with its eigenvalues valued in $[\lambda,\Lambda]$ almost everywhere. 
In addition, we will assume that $b\in L^\infty\cap H^1$ is \emph{nondegenerate} in the sense that 
there exist some constants $K>0$ and $\alpha\in (0,1]$ such that 
\begin{equation}\label{nond}
	\forall \mu\in\mathbb{R},\  \forall \nu\in\mathbb{S}^{{d_1}-1}, \ \forall \epsilon>0,\ \  
	\big|\{ v\in B_1(0): | \mu+b(v)\cdot\nu|\le \epsilon\}\big| 
	\le K\epsilon^\alpha. 
\end{equation}
Such weak conditions are sufficient to derive the local boundedness of solutions to \eqref{FP}; see Theorem~\ref{bdd}. 
To enhance the regularity of the solutions to obtain Theorem \ref{regularity}, rescaling each open ball in the above condition is necessary. In this way, a stronger assumption on $b\in C^1$ reads as follows: 
there exists some constant $K>0$ such that for any open ball $B_r(v_0)$ with $B_{2r}(v_0)\subset U_v$, the following holds:  
 \begin{equation}\label{N}
 \forall \mu\in\mathbb{R},\ \ \forall \nu\in\mathbb{S}^{{d_1}-1}, \ \forall \epsilon>0,\ \  
 \big|\{ v\in B_r(v_0): | \mu+b(v)\cdot\nu|\le \epsilon\}\big| 
 \le Kr^{{d_2}-1}\epsilon. 
 \end{equation}
We will see that, when $d_1=d_2$, this condition for $b\in C^1$ is equivalent to the assumption that the induced matrix norm of $(Db)^{-1}$ is bounded (see Lemma~\ref{nondlemma}).

\subsection{Main results}
Before stating the main contributions of this article, let us make the notion of weak solutions precise. 
\emph{Weak solutions} to \eqref{FP} in $[t_1,t_2]\times U_x\times U_v$ 
are defined as functions $f\in L^\infty_t\left([t_1,t_2],L^2_{x,v}(U_x\times U_v)\right)\cap L^2_{t,x}\left([t_1,t_2]\times U_x,H^1_v(U_v)\right)$ satisfying  
$(\partial_t+b\cdot\nabla_x)f\in L^2_{t,x}\left([t_1,t_2]\times U_x,H^{-1}_v(U_v)\right)$
and verifying \eqref{FP} in the sense of distributions.
Under the same assumptions on $f$ with the equality of \eqref{FP} replaced by the inequality $\le$  that holds when acting on any nonnegative test functions, we say that $f$ is a \emph{subsolution} to \eqref{FP}. 

\begin{theorem}[local boundedness]\label{bdd}
	Let $f$ be a subsolution to \eqref{FP} satisfying \eqref{H} and \eqref{nond} in $Q_1:=(-1,0]\times B_1\times B_1$, and in addition, 
	$$\| b\|_{L^\infty(B_1)}+\| Db\|_{L^2(B_1)}\le\Lambda {\quad and\quad} s\in L^q(Q_1)$$ 
	for some $q>\frac{(4+\alpha)(4+d_2)(1+{d_1}+{d_2})}{2\alpha}$. Then, the positive part $f^+$ of the subsolution is bounded in $Q_{\rm int}:=(-\frac{1}{2},0]\times B_\frac{1}{2}\times B_\frac{1}{2}$. 
	More precisely, there exists some positive constant $C$ only depending on $\lambda,\Lambda,d_1,d_2,K,\alpha,$ and $q$ such that  
	\begin{equation*}\label{bddestimate}
	\sup\nolimits_{Q_{\rm int}} f^+
	\le C \left(\| f^+\|_{L^2(Q_1)}+\| s\|_{L^q(Q_1)}\right). 
	\end{equation*}
\end{theorem}

\begin{theorem}[H\"older regularity]\label{regularity}
	Let $d_1=d_2=d$ and $f$ be a weak solution to \eqref{FP} satisfying \eqref{H} and \eqref{N} in $Q_1$.
 Assume $b\in C^1(B_1)$ with a function $o_1:[0,\infty)\rightarrow[0,\infty)$ as the modulus of continuity of $Db$, that is, 
	\begin{equation*}
	|Db(v)-Db(w)|\le o_1(|v-w|)  {\ \ \rm for\ any\ }v,w\in B_1 {\quad and\quad}\lim_{r\rightarrow0}o_1(r)=o_1(0)=0.
	\end{equation*} 
	In addition, we assume
	$$\| b\|_{L^\infty(B_1)}+\| Db\|_{L^\infty(B_1)}\le\Lambda {\quad and\quad} s\in L^q(Q_1)$$ 
	for some $q>(1+2d)^2$.  
	Then, there exist some constants $\beta\in(0,1)$ and $C>0$ only depending on $\lambda,\Lambda,d,K,q,$ and $o_1$ 
	such that 
	\begin{equation*}
	\| f \|_{C^{\beta}(Q_{\rm int})}
	\le C\left(\| f\|_{L^2(Q_1)}+\| s\|_{L^q(Q_1)}\right). 
	\end{equation*}
\end{theorem}

\subsection{Related work and contributions}
The boundedness and Hölder regularity results above for the free streaming case, that is, $b(v)=v$, were first achieved in \cite{PP} and \cite{WZ}, respectively. 
A simplified proof was provided in \cite{GIMV}, where the authors combined the methods arising from velocity averaging with De Giorgi-Nash-Moser theory to derive the Hölder continuity of weak solutions to the kinetic Fokker-Planck equation \eqref{FP} with $b(v)=v$. 
As De Giorgi's method has extended to nonlocal parabolic operators (see \cite{CV}, \cite{CCV}, \cite{FK}), kinetic models with integral diffusion were studied in \cite{IS} and \cite{St}. 

On the one hand,  the arguments in \cite{PP} and \cite{WZ} rely on the explicit expression of fundamental solutions
to Kolmogorov type operators including $\partial_t+v\cdot\nabla_x-\Delta_v$. 
On the other hand, the fundamental solution of the operator $\partial_t+b(v)\cdot\nabla_x-\Delta_v$ cannot be computed explicitly in general, especially for $b(v)$ that is not affine. 
The main contribution of the present paper is to show that in the setting of general transport operators $\partial_t+b(v)\cdot\nabla_x$, velocity averaging lemmas yield local boundedness and Hölder regularity of solutions under a suitable nondegeneracy condition on the rough vector field $b(v)$.

We revisit the averaging lemmas and the strategy proposed in \cite{GIMV} to remedy the lack of regularity of the diffusion coefficients and the degeneracy of the elliptic structure with a nondegenerate $b(v)$. 
Due to the efforts to prove the averaging lemmas (see the discussion in \S\ref{nondegenerate} below), for the boundedness part, $b(v)$ is only required to be nondegenerate in the sense of \eqref{nond} and lying in $L^\infty\cap H^1$. To obtain the Hölder regularity of solutions, we have to enhance $b(v)\in C^1$ to be nondegenerate in the sense of \eqref{N}, which is equivalent to assuming that $Db(v)$ is invertible everywhere.  

Besides, in \cite{GIMV}, two methods toward the local boundedness of solutions to \eqref{FP} with $b(v)=v$ are presented. One is Moser's approach in the case without source term and the other one is De Giorgi's approach. One of our contributions in this part is using Moser's iterative method to deal with the general type equation \eqref{FP} in the appearance of source terms.

Finally, a useful result of propagation of Hölder regularity is derived by means of some zero extension. 
The extension reduces the boundary estimate to the same regularization procedure as the interior one. 
The reduction argument whose idea arises from elliptic theory (see, for instance, \cite[Theroem~8.27]{GT}) can also be applied to obtain H\"older estimates for solutions with prescribed H\"older continuous boundary values near a given Lipschitz boundary.

\subsection{Motivation and background}\label{backgrounds}
\subsubsection{Physical interpretation}
With $d_1=d_2=d$, some physical interpretation for the Fokker-Planck equation \eqref{FP} comes from considering the Liouville transport operator 
\begin{equation*}
L:=\partial_t+\{\cdot,H\}=\partial_t+b\cdot\nabla_x-B\cdot\nabla_v,
\end{equation*} 
where $H$ is a given Hamiltonian; then the velocity term and the force term are given by  $b=\nabla_vH$ and $B=\nabla_xH$, respectively.  For instance, 
$b(v)= v$ in the setting of classical Newton's laws, while $b(v)= \frac{v}{\sqrt{1+|v|^2}}$ is induced by the relativistic effect. 
It is not hard to see that \eqref{nond} and \eqref{N} both hold with $\alpha=1$ for $b(v)=v$ or $\frac{v}{\sqrt{1+v^2}}$ in any bounded set $U_v$. 
Varieties of nonlinear coefficients $b(v)$ also appear in quantum transport of electrons in semiconductors, where the associated semiclassical Hamiltonian $H=\mathcal{E}(v)+V(t,x)$, $\mathcal{E}(v)$ is given by the energy band diagram of the semiconductor, and $V(t,x)$ is the varying potential contribution; see \cite{DM}, \cite{MRS} for further details. 

In a large particle system, if there is no collision, then the Liouville equation reads $Lf=0.$ 
It is a first order hyperbolic equation that models the evolution of the distribution function $f(t,x,v)$ for the system of identical particles, that is the density of particles at time $t$ located at the position $x$ with a physical state described by the velocity variable $v$. One has to modify the equation by adding a collision term $\mathscr{C}(f)$ on the right-hand side when the effect of the collisions between particles in a certain surrounding bath is taken into account, which means $Lf=\mathscr{C}(f).$
Different models lead to various collision operators $\mathscr{C}$. For instance, under Coulomb collisions, the Landau-Coulomb operator proposed in \cite{Lan} has a quasilinear form and enjoys the ellipticity in velocity as long as controls of some macroscopic physical quantities are provided; see \cite{DV}. More kinetic diffusive models can be found in \cite{Ch}, \cite{Vi}. 
In view of applications to these nonlinear equations, the structure of mixing transport and degenerate diffusion inspires the study of the Fokker-Planck equation \eqref{FP} with rough coefficients.

\subsubsection{Nondegenerate condition and velocity averaging}\label{nondegenerate}
The nondegenerate condition \eqref{N} we impose arises from velocity averaging and is used in the zooming procedure (see subsections \ref{zoom}, \ref{invertibility}). 

Recall the form of our equation \eqref{FP}. Due to its hyperbolicity, we cannot expect the gain of regularity for the solution $f(t,x,v)$ itself of the kinetic transport equation 
\begin{equation*}\label{tranport}
\left(\partial_t+b(v)\cdot\partial_x\right)f(t,x,v)=S(t,x,v),
\end{equation*}
where the source term $S\in L_{t,x}^2(H^{-m}_v)$ with $m=0,1$ is rough. 
However, from the perspective of the symbol of the transport operator, it is elliptic outside small neighborhoods around its characteristic directions. The lack of regularity inside can be compensated by the ellipticity outside on a macroscopic level. 
This kind of phenomenon was first observed by Agoshkov \cite{Ag}, Golse, Perthame and Sentis \cite{GPS} independently and formulated precisely later on in \cite{GLPS}. 
It was also shown in \cite{GLPS} that the condition \eqref{nond} ensures that a velocity averaging lemma is satisfied when $m=0$, which predicts that there is some better regularity on the averages of $f$ with respect to the microscopic variable $v$. 
A subsequent improvement in the free streaming case ($b(v)=v$) with loss of derivatives in the source ($m>0$) was obtained in \cite{DPL}. We point out that the loss of $v$-derivatives in the source requires the $v$-regularity on the drift coefficient $b(v)$. 
As the regularizing effect of general transport operators comes from their nondegeneracy, we derive a similar result (Lemma \ref{VA}) for general nondegenerate $b(v)$ lying in $L^\infty\cap H^1$. 

We remark that the condition \eqref{nond} is also a quantitative version of the noncharacteristic condition assumed in \cite{GG}, where microlocal analysis techniques are used and high order smoothness is required on the coefficient $b$. Besides, in the case  $d_2=1$, the nondegeneracy of $b$ corresponds to the genuine nonlinearity of hyperbolic equations in the kinetic formulation of scalar conservation laws; see \cite{LPT}. 
More details about the conditions \eqref{nond} and \eqref{N} in an analysis viewpoint can be found in Remark~\ref{rk_NG}, Remark~\ref{rkNA}, and subsection~\ref{invertibility} below. 

Combining the velocity averaging with an a priori bound on the derivatives in velocity of solutions, a hypoelliptic transfer of compactness was observed and used in \cite{Lions}. In \cite{Bo}, F.~Bouchut put the transfer mechanism in the form of precise fractional Sobolev regularity estimates. 
For \eqref{FP}, an a priori bound on $v$-derivatives of solutions is derived from the energy estimate due to the diffusion term. 
The mixing effect of transport and diffusion allows us to see a hypoelliptic structure of the equation. 

\subsubsection{Hypoelliptic structure}\label{hypoelliptic}
We remark that the structure of the Fokker-Planck equation has two aspects, as it is the combination of a transport operator and an elliptic operator in the velocity variable. 

On the one hand, it is a degenerate parabolic equation, as it only satisfies the elliptic condition with respect to the variable $v$. It is well-known that elliptic operators enjoy the regularizing effect even if their coefficients are rough. 
In the late 1950s, De~Giorgi \cite{DG} obtained the H\"{o}lder regularity of solutions to elliptic equations in divergence form with bounded measurable coefficients and Nash \cite{Na} independently solved the case of parabolic equations. 
In the spirit of De~Giorgi's idea to achieve the local boundedness of solutions, Moser \cite{Mo1} provided a new proof by developing a technique now known as Moser's iteration. 

On the other hand, when $b(v)$ satisfies the nondegenerate condition \eqref{N}, \eqref{FP} has a kind of hypoelliptic structure. 
It was first noticed by Kolmogorov \cite{Ko} that the operator $L_K:=\partial_x^2+x\partial_y-\partial_t$ in three-dimensional space-time preserves singular supports, which means that the smoothness of $L_K f$ implies the smoothness of $f$, as its fundamental solution can be calculated explicitly. 
It inspired Hörmander \cite{Ho} to study general hypoelliptic second order differential operators including the ones of the form $\sum_{i=1}^m X_i^*X_i+X_0$ associated with a system of real smooth vector fields $\{X_i\}_{i=0}^m$, where $X_i^*$ is the formal adjoint of $X_i$. Hörmander's celebrated theorem states that if the Lie algebra generated by $\{X_i\}_{i=0}^m$ spans the full tangent space at each point, then the operator preserves singular supports. 
Without regard to the smoothness of their coefficients, setting $(X_1,X_2,\ldots,X_{d_2})^T:=\sqrt{A}\nabla_v$ and $X_0:=\partial_t+b(v)\cdot\nabla_x$ satisfying \eqref{N}, we point out that Hörmander's condition holds for this system $\{X_i\}_{i=0}^{d_2}$. 
More detailed relations between the nondegenerate condition \eqref{N} and Hörmander's condition will be discussed in Subsection~\ref{invertibility}. 

\subsection{Notation}
We denote by $C$ a \emph{universal} constant if it depends only on $\lambda$, $\Lambda$, $d_1$, $d_2$, $q$, $K$, $\alpha$, and $o_1$. 
We will write $X \lesssim Y$ to say that $X \le C Y$ for some universal constant $C>0$. 
For simplicity, we do not change its symbol $C$, even if it changes line by line. 

We define the cylinder domain  $Q_R:=(-R^2,0]\times B_{R^3}(0)\times B_{R}(0)$. We will occasionally simplify $B_{R}(0)$ by $B_R$. 

Throughout the article, $\sup f$ and $\inf f$ refer to the essential supremum and the essential infimum of $f$ in a given domain, respectively. 

\subsection{Organization of the paper}
In Section~\ref{sectionva}, we establish velocity averaging lemmas for general transport equations. 
In Section~\ref{sectionbdd}, the local boundedness Theorem~\ref{bdd} is derived from velocity averaging lemmas and the De Giorgi-Nash-Moser theory. 
Section~\ref{sectionholder} is devoted to H\"older regularity, including the proof of Theorem~\ref{regularity} and a boundary estimate. 

\medskip\noindent\textbf{Acknowledgement.} The author would like to thank François Golse and Cyril Imbert for suggesting the problem and for helpful discussions 
and also thanks the anonymous referee for suggestions to improve the manuscript. 
\includegraphics[height=0.3cm]{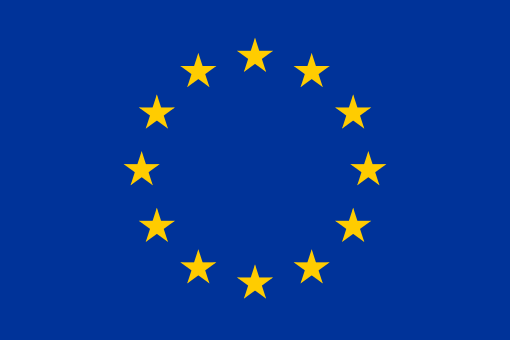} 
This work has received funding from the European Union’s Horizon 2020 research and innovation programme under the Marie Skłodowska-Curie grant agreement No 754362.

\section{Velocity averaging}\label{sectionva}
A key ingredient of the proof in the work of Golse et al. \cite{GIMV} is based on the velocity averaging method and relies on the result due to Bouchut \cite{Bo}, which gives a quantitative hypoelliptic regularity estimate.

In this section, we first generalize the classical velocity averaging lemma to the case with general nondegenerate $b(v)$. 
Armed with the mechanism of the averaging theory, we then derive the regularity in all variables on the solution in view of the extra regularity with respect to $v$ of the solution. 

\subsection{Classical averaging theory}
The following result was obtained in \cite{DPL} and \cite{Re} in the cases where $b(v)=v$ and $b(v)= \frac{v}{\sqrt{1+|v|^2}}$, respectively. 
The main strategy of the proof is similar to them in spirit, which is based on a microlocal decomposition in the Fourier space of time and space. 
By this splitting and the nondegenerate condition, we can handle the terms in the region around the characteristic directions of the transport operator and its complement. 
Then, in appropriate microlocal domains, we will exploit some ellipticity in the complement to compensate the lack of regularity of the transport operator around its characteristic directions in the averaging sense. 

\begin{lemma}\label{VA}
Consider the open ball $B_R\subset \mathbb{R}^{d_2}$ and let $\psi(v)\in C_c^1 (B_R)$. Assume that $h$, $g_0\in L^2 (\mathbb{R}_t\times\mathbb{R}^{d_1}_x\times B_R)$ and 
 $g_1\in L^2 (\mathbb{R}_t\times\mathbb{R}^{d_1}_x\times B_R,\mathbb{R}^{d_2})$ verify the transport equation
\begin{equation*}\label{transport}
\left(\partial_t+b(v)\cdot\nabla_x\right)h(t,x,v) = {\rm div}_v g_1 (t,x,v)+g_0(t,x,v)
\end{equation*}
in the sense of distributions, where $b(v)\in L^\infty\cap H^1 (B_R,\mathbb{R}^{d_1})$ satisfies the nondegenerate condition \eqref{nond} (replace $B_1(0)$ by $B_R$). Then, the average satisfies
\begin{equation*}\label{averaging}
h_\psi(t,x) :=
\int_{B_R} h(t,x,v)\psi(v) \dif v
\in H^{\varsigma}(\mathbb{R}_t\times\mathbb{R}^{d_1}_x)
\end{equation*}
with the estimate
\begin{equation*}\label{VAest}
\| h_\psi\|_{{H}^{\varsigma}}\le C_b
\| \psi\|_{W^{1,\infty}}\left(\| h\|_{L^2}+\| g_1\|_{L^2}+\| g_0\|_{L^2}\right),
\end{equation*}
where $\varsigma:=\frac{\alpha}{4+\alpha}$ and $C_b>0$ only depending on $d_1$, $d_2$, $K$, $R$, $\| b\|_{L^\infty}$, and $\| Db\|_{L^2}$.
\end{lemma}

\begin{notation}
We always omit the notation of domain if the norm is over $(t,x,v)\in\mathbb{R}_t\times\mathbb{R}^{d_1}_x\times B_R$ or over $v\in B_R$. 
\end{notation}

\begin{remark}\label{rk_NG}
The nondegenerate condition \eqref{nond} roughly asserts that $b(v)$ cannot concentrate on any hyperplane, 
which forces it to move in every direction at each point. Otherwise, there will be no regularity information in $x$ derived from the equation about $h$ as well as $h_\psi$ on some sets of positive  measure, except along the normal direction of the supposed hyperplane. 
\end{remark}

\begin{remark}\label{rklip}
If $b(v)\in W^{1,\infty}(B_R,\mathbb{R}^{d_1})$, then by the similar strategy in the proof below with even simpler computation (change the parameter $m(\xi)$ by picking $m=1$ for $|\xi|\le1$ and $m=|\xi|^\frac{1}{2}$ for $|\xi|>1$), the conclusion in the above lemma can be improved to bound $\| h_\psi\|_{H^\frac{\alpha}{4}}$ with $C_b$ depending only on $d_1$, $d_2$, $K$, $R$, and $\| b\|_{W^{1,\infty}}$. 
\end{remark}

\begin{proof}
We split the proof into three steps as follows.
	
\noindent{\textbf{Step 1.}} The microlocal decomposition. \\
By performing the Fourier transform $\hat{}$ with respect to $(t,x)\rightarrow(\tau,\xi)$, we have
\begin{equation}\label{norm-h}
\| h_\psi\|_{\dot{H}^{\varsigma}}^2 =
\int_{\mathbb{R}\times\mathbb{R}^{d_1}}\big| I(\tau,\xi) \big|^2\left(|\tau|+ |\xi|\right)^{2\varsigma}\dif\tau\dif\xi,
\end{equation}
where the constant $\varsigma\in(0,1)$ is to be determined and $I(\tau,\xi)$ is defined by 
\begin{equation*}
I(\tau,\xi):=\int_{\mathbb{R}^{d_1}} \hat{h}(\tau,\xi,v)\psi(v) \dif v. 
\end{equation*}

Moreover, $\hat{h}$ satisfies the equation
\begin{equation}\label{VArelation}
{\rm i}\left(\tau+b(v)\cdot\xi\right)\hat{h}={\rm div}_v\hat{g}_1+\hat{g}_0.
\end{equation}
Since it is effective only when $\tau+b(v)\cdot\xi\neq 0$, we introduce a cut-off function $\zeta\in C^\infty_c(\mathbb{R})$ such that $0\le\zeta\le 1$, $\zeta|_{[-\frac{1}{2},\frac{1}{2}]}\equiv 1$, and $\zeta|_{[-1,1]^c}\equiv 0$. Write
\begin{equation*}
I(\tau,\xi)=I_1(\tau,\xi)+I_2(\tau,\xi)
\end{equation*} 
with 
\begin{equation*}
I_1(\tau,\xi):=\int_{B_R} \hat{h}(\tau,\xi,v)\psi(v)\zeta\left(\frac{\tau+b(v)\cdot\xi}{m}\right)\dif v ,
\end{equation*} 
\begin{equation*}
I_2(\tau,\xi):=\int_{B_R} \hat{h}(\tau,\xi,v)\psi(v)\left[1-\zeta\left(\frac{\tau+b(v)\cdot\xi}{m}\right)\right]\dif v,
\end{equation*} 
where $m=m(\xi)\ge 1$ to be determined. 

By multiplying the relation \eqref{VArelation} by $\psi(v)$ and integrating by parts, we have
\begin{equation}\label{I2w}
I_2(\tau,\xi)={\rm i}\int_{B_R}
\hat{g}_1(\tau,\xi,v)\cdot \nabla_v \left[\Phi(\tau,\xi,v)\psi(v)\right]
-\hat{g}_0(\tau,\xi,v)\Phi(\tau,\xi,v)\psi(v)\dif v,
\end{equation}
where we set
\begin{equation}\label{w(v)}
\Phi(\tau,\xi,v):=\frac{1}{\tau+b(v)\cdot\xi}\left[1-\zeta\left(\frac{\tau+b(v)\cdot\xi}{m}\right)\right].
\end{equation}

\medskip\noindent{\textbf{Step 2.}} The estimates of $I_1$ and $I_2$. \\
For $I_1(\tau,\xi)$, the integrand is vanishing whenever $|\tau+b(v)\cdot\xi|>m$ 
and hence $I_1(\tau,\xi)=0$ on $\{|\tau|> m+\overline{b}|\xi|\}$, where we abbreviate $\overline{b}:=\| b\|_{L^\infty(B_R)}$. Combining with the nondegenerate condition \eqref{nond} yields
\begin{equation*}
\int_{B_R} \zeta^2\left(\frac{\tau+b(v)\cdot\xi}{m}\right)\dif v\le K\left(\frac{m}{|\xi|}\right)^\alpha\mathbbm{1}_{|\tau|\le m+\overline{b}|\xi|}. 
\end{equation*}
Thus, we obtain
\begin{align}\label{I1}
|I_1(\tau,\xi)|^2 
\le\| \psi\|_{L^\infty}^2 \|\hat{h}\|_{L^2_v}^2 \bigg\|\zeta\left(\frac{\tau+b(v)\cdot\xi}{m}\right)\bigg\|_{L^2_v}^2
\le K\|\psi\|_{L^\infty}^2 \|\hat{h}\|_{L^2_v}^2 \left(\frac{m}{|\xi|}\right)^\alpha \mathbbm{1}_{|\tau|\le m+\overline{b}|\xi|}. 
\end{align}

For $I_2(\tau,\xi)$, observing \eqref{I2w} and \eqref{w(v)}, we have
\begin{equation}\label{I12}
|I_2(\tau,\xi)|^2 \le 
\left(\|\hat{g}_1\|_{L^2_v}^2+\|\hat{g}_0\|_{L^2_v}^2\right)
\left(\|\nabla\psi\|_{L^{\infty}}^2\|\Phi\|_{L^2_v}^2+ \|\psi\|_{L^{\infty}}^2\|\nabla_v\Phi\|_{L^2_v}^2\right) . 
\end{equation} 
By simple computation, 
\begin{equation*}
|\Phi|^2+|\nabla_v \Phi|^2
\lesssim\frac{1}{|\tau+b(v)\cdot\xi|^2}
\left[1+\left(\frac{|\xi||Db(v)|}{m}\right)^2\right]\mathbbm{1}_{|\tau+b(v)\cdot\xi|\ge\frac{m}{2}},
\end{equation*}
so that
\begin{align}\label{I2int}
\|\Phi\|_{L^2_v}^2+\|\nabla_v\Phi\|_{L^2_v}^2
&\lesssim \int_{B_R}
\left[1+\left(\frac{|\xi||Db(v)|}{m}\right)^2\right]\mathbbm{1}_{|\tau+b(v)\cdot\xi|\ge\frac{m}{2}}\frac{\dif v}{|\tau+b(v)\cdot\xi|^2}\nonumber\\
&= \int_{B_R}\dif v
\int^\infty_{|\tau+b(v)\cdot\xi|}
\left[1+\left(\frac{|\xi||Db(v)|}{m}\right)^2\right]\mathbbm{1}_{|\tau+b(v)\cdot\xi|\ge\frac{m}{2}}\frac{2\dif s}{s^3}\nonumber\\
&=\int_{\frac{m}{2}}^\infty \frac{2\dif s}{s^3}
\int_{B_R}\left[1+\left(\frac{|\xi||Db(v)|}{m}\right)^2\right]\mathbbm{1}_{\frac{m}{2}\le|\tau+b(v)\cdot\xi|\le s}\dif v.
\end{align}
Using the nondegenerate condition \eqref{nond}, we have 
\begin{align}\label{I21int}
\|\Phi\|_{L^2_v}^2+\|\nabla_v\Phi\|_{L^2_v}^2
\lesssim \int_{\frac{m}{2}}^\infty
\left[K\left(\frac{s}{|\xi|}\right)^\alpha + \frac{\overline{b'}^2|\xi|^2}{m^2}\right] \frac{\dif s}{s^3} 
\le C_b\left(\frac{m^{\alpha-2}}{|\xi|^\alpha} + \frac{|\xi|^2}{m^4}\right),
\end{align}
where we denote $\overline{b'}:=\|Db\|_{L_v^2(B_R)}$. 
We will use this estimate by setting $m=1$ when $\xi$ is near the origin. For large $|\xi|$, we are going to pick $m$ to be $|\xi|$ to some positive power. 
Besides, we have a more precise estimate when $|\tau|\ge m+\overline{b}|\xi|$, as the origin is away from the domain of the integration in \eqref{I2int} automatically. In this case, 
\begin{align}\label{I22int}
\|\Phi\|_{L^2_v}^2+\|\nabla_v\Phi\|_{L^2_v}^2
&\lesssim \int_{|\tau|-\overline{b}|\xi|}^{|\tau|+\overline{b}|\xi|} 
\left[K\left(\frac{s}{|\xi|}\right)^\alpha + \frac{\overline{b'}^2|\xi|^2}{m^2}\right] \frac{\dif s}{s^3}
+\int_{|\tau|+\overline{b}|\xi|}^{\infty}\left(|B_R|+\frac{\overline{b'}^2|\xi|^2}{m^2}\right)\frac{\dif s}{s^3} \nonumber\\
&\le C_b\left(\frac{|\tau|^{1-\alpha}|\xi|^{1-\alpha}}{(|\tau|^2-\overline{b}^2|\xi|^2)^{2-\alpha}} + \frac{|\tau||\xi|^3}{m^2(|\tau|^2-\overline{b}^2|\xi|^2)^2}
+\frac{m^2+|\xi|^2}{m^2(|\tau|^2+|\xi|^2)}\right),
\end{align}
where we applied 
\begin{equation*}
\left(|\tau|+\overline{b}|\xi|\right)^{2-\alpha}-\left(|\tau|-\overline{b}|\xi|\right)^{2-\alpha}\le 4\overline{b}|\tau|^{1-\alpha}|\xi|
\end{equation*}
in the last inequality, as we notice that $|\tau|>\overline{b}|\xi|$ and $\alpha\in (0,1]$.
It then follows from \eqref{I12}, \eqref{I21int}, and \eqref{I22int} that
\begin{align}\label{I2}
|I_2(\tau,\xi)|^2
\le& \|\psi\|_{W^{1,\infty}}^2\left(\|\hat{g}_1\|_{L_v^2}^2+\|\hat{g}_0\|_{L_v^2}^2\right)
\left(\|\Phi\|_{L^2_v}^2+\|\nabla_v\Phi\|_{L^2_v}^2\right)
\left(\mathbbm{1}_{|\tau|\le m+\overline{b}|\xi|}+\mathbbm{1}_{|\tau|\ge m+\overline{b}|\xi|}\right)\nonumber\\
\le& C_b\|\psi\|_{W^{1,\infty}}^2\left(\|\hat{g}_1\|_{L_v^2}^2+\|\hat{g}_0\|_{L_v^2}^2\right)
\Bigg[\left(\frac{m^{\alpha-2}}{|\xi|^\alpha} + \frac{|\xi|^2}{m^4}\right)\mathbbm{1}_{|\tau|\le m+\overline{b}|\xi|}\nonumber\\
&+ 
\left(\frac{|\tau|^{1-\alpha}|\xi|^{1-\alpha}}{(|\tau|^2-\overline{b}^2|\xi|^2)^{2-\alpha}} + \frac{|\tau||\xi|^3}{m^2(|\tau|^2-\overline{b}^2|\xi|^2)^2}
+\frac{m^2+|\xi|^2}{m^2(|\tau|^2+\overline{b}^2|\xi|^2)} \right)
\mathbbm{1}_{|\tau|\ge m+\overline{b}|\xi|}\Bigg].
\end{align}

Combining \eqref{I1} with \eqref{I2} and recalling $m\ge 1$, we conclude that
\begin{align}\label{optim}
|I(\tau,\xi)|^2
\le & C_b\|\psi\|_{W^{1,\infty}}^2 G(\tau,\xi)
\Bigg[\left(\frac{m^{\alpha}}{|\xi|^\alpha} + \frac{|\xi|^2}{m^4}\right)\mathbbm{1}_{|\tau|\le m+\overline{b}|\xi|}\nonumber\\
&+ \left(\frac{|\tau|^{1-\alpha}|\xi|^{1-\alpha}}{(|\tau|^2-\overline{b}^2|\xi|^2)^{2-\alpha}} + \frac{|\tau||\xi|^3}{m^2(|\tau|^2-\overline{b}^2|\xi|^2)^2}
+\frac{m^2+|\xi|^2}{m^2(|\tau|^2+\overline{b}^2|\xi|^2)}\right)\mathbbm{1}_{|\tau|\ge m+\overline{b}|\xi|}\Bigg], 
\end{align}
where we abbreviate $G(\tau,\xi):=\|\hat{h}\|_{L_v^2}^2+\|\hat{g}_1\|_{L_v^2}^2+\|\hat{g}_0\|_{L_v^2}^2$. 

\medskip\noindent{\textbf{Step 3.}} Tracking in appropriate microlocal domains. \\
For $|\xi|\le 1$, we pick $m=1$ in the estimate \eqref{optim} so that
\begin{align}\label{Ixi1}
|I(\tau,\xi)|^2
\le&  C_b\|\psi\|_{W^{1,\infty}}^2 G(\tau,\xi)
\Bigg[\left(|\xi|^{-\alpha}+|\xi|^2\right)\mathbbm{1}_{|\tau|\le 1+\overline{b}|\xi|} \nonumber\\
&+ \left(\frac{|\tau|^{1-\alpha}|\xi|^{1-\alpha}}{(|\tau|^2-\overline{b}^2|\xi|^2)^{2-\alpha}} + \frac{|\tau||\xi|^3}{(|\tau|^2-\overline{b}^2|\xi|^2)^2}
+\frac{1+|\xi|^2}{|\tau|^2+\overline{b}^2|\xi|^2} \right)\mathbbm{1}_{|\tau|\ge 1+\overline{b}|\xi|}\Bigg].
\end{align}
As for $|\xi|\ge 1$, we pick $m=|\xi|^{\varrho}$ with $\varrho\in(0,1)$ to be determined so that
\begin{align}\label{Ixi2}
|I(\tau,\xi)|^2
\le&  C_b\|\psi\|_{W^{1,\infty}}^2 G(\tau,\xi)
\Bigg[\left(|\xi|^{\alpha(\varrho-1)}+|\xi|^{2(1-2\varrho)}\right)\mathbbm{1}_{|\tau|\le |\xi|^\varrho+\overline{b}|\xi|} \nonumber\\
&+ \left(\frac{|\tau|^{1-\alpha}|\xi|^{1-\alpha}}{(|\tau|^2-\overline{b}^2|\xi|^2)^{2-\alpha}} + \frac{|\tau||\xi|^{3-2\varrho}}{(|\tau|^2-\overline{b}^2|\xi|^2)^2}
+\frac{1+|\xi|^{2-2\varrho}}{|\tau|^2+\overline{b}^2|\xi|^2} \right)\mathbbm{1}_{|\tau|\ge |\xi|^\varrho+\overline{b}|\xi|}\Bigg].
\end{align}
Recalling \eqref{norm-h}, we split the quantity $\| h_\psi\|^2_{\dot{H}^{\varsigma}}=\Sigma_{i=1}H_i$ as the following ordinal five terms: 
\begin{equation*}
\begin{split}
H_1+H_2+H_3+H_4+H_5:=
\int\big|I\big|^2|\tau|^{2\varsigma}\mathbbm{1}_{|\xi|\le1}\mathbbm{1}_{|\tau|\le 1+\overline{b}}\dif\tau\dif\xi
+\int\big|I\big|^2|\tau|^{2\varsigma}\mathbbm{1}_{|\xi|\le1}\mathbbm{1}_{|\tau|>1+\overline{b}}\dif\tau\dif\xi&\\
+\int\big|I\big|^2|\xi|^{2\varsigma}\mathbbm{1}_{|\xi|\le 1}\dif\tau\dif\xi
+\int\big|I\big|^2|\tau|^{2\varsigma}\mathbbm{1}_{|\xi|> 1}\dif\tau\dif\xi
+\int\big|I\big|^2|\xi|^{2\varsigma}\mathbbm{1}_{|\xi|>1}\dif\tau\dif\xi&.
\end{split}
\end{equation*}
Observing $I(\tau,\xi)\le\|\hat{h}(\tau,\xi,\cdot)\|_{L^2}\|\psi\|_{L^2}$ due to H\"older's inequality, we know that $H_1$ and $H_3$ can be controlled by $C_b\| h\|_{L^2}^2\|\psi\|_{L^2}^2$. 

As for $H_2$, using \eqref{Ixi1}, we get
\begin{equation*}
\begin{split}
H_2
\le& C_b\|\psi\|_{W^{1,\infty}}^2\int G(\tau,\xi)
\mathbbm{1}_{|\tau|\ge 1+\overline{b}}\mathbbm{1}_{|\xi|\le 1}\\
&\cdot\left(\frac{|\tau|^{1-\alpha}|\xi|^{1-\alpha}}{(|\tau|^2-\overline{b}^2|\xi|^2)^{2-\alpha}} 
 + \frac{|\tau||\xi|^3}{(|\tau|^2-\overline{b}^2|\xi|^2)^2}
 +\frac{1+|\xi|^2}{|\tau|^2+\overline{b}^2|\xi|^2} \right)
|\tau|^{2\varsigma}\dif\tau\dif\xi .
\end{split}
\end{equation*}
Notice that the integral above is effective only over $\{|\xi|\le 1\}\cap\{|\tau|\ge1+\overline{b}\}$. For any fixed $|\xi|\le 1$, the functions 
\begin{equation}\label{VAfunction}
\frac{|\tau|^{1-\alpha}|\tau|^{2\varsigma}}{(|\tau|^2-\overline{b}^2|\xi|^2)^{2-\alpha}}  \quad{\rm and}\quad
\frac{|\tau||\tau|^{2\varsigma}}{(|\tau|^2-\overline{b}^2|\xi|^2)^2}
\end{equation}
are both decreasing with respect to $|\tau|$, so they are bounded by the values that they take at $|\tau|=1+\overline{b}$. 
Also, the left function in the integrand
\begin{equation*}
\frac{\big(1+|\xi|^2\big)|\tau|^{2\varsigma}}{|\tau|^2+\overline{b}^2|\xi|^2}
\le \frac{2|\tau|^{2\varsigma}}{|\tau|^2}
\le 2
\end{equation*}
in the effective set and hence the term  $H_2$ can be estimated.

For the remaining terms $H_4$ and $H_5$, it suffices to check that the powers of the multipliers $|\tau||\xi|$ in the integrands are all nonpositive. More precisely, by \eqref{Ixi2}, 
\begin{equation*}
\begin{split}
H_4
\le&  C_b\|\psi\|_{W^{1,\infty}}^2\int G(\tau,\xi)
\Bigg[\left(|\xi|^{\alpha(\varrho-1)}+|\xi|^{2(1-2\varrho)}\right)\mathbbm{1}_{|\tau|\le |\xi|^\varrho+\overline{b}|\xi|} \\
&+ \left(\frac{|\tau|^{1-\alpha}|\xi|^{1-\alpha}}{(|\tau|^2-\overline{b}^2|\xi|^2)^{2-\alpha}} + \frac{|\tau||\xi|^{3-2\varrho}}{(|\tau|^2-\overline{b}^2|\xi|^2)^2}
+\frac{1+|\xi|^{2-2\varrho}}{|\tau|^2+\overline{b}^2|\xi|^2} \right)\mathbbm{1}_{|\tau|\ge |\xi|^\varrho+\overline{b}|\xi|}\Bigg] |\tau|^{2\varsigma}\mathbbm{1}_{|\xi|\ge 1}\dif\tau\dif\xi.
\end{split}
\end{equation*}
We choose $\varsigma$ and $\varrho$ verifying $\alpha(\varrho-1)+2\varsigma=0$ and $2(1-2\varrho)+2\varsigma=0$, which gives 
$$\varsigma=\frac{\alpha}{4+\alpha}\quad{\rm and}\quad \varrho=\frac{2+\alpha}{4+\alpha}.$$ 
Besides, the two functions defined in \eqref{VAfunction} appearing in the above integrand are controlled by the values that they take at $|\tau|=|\xi|^\varrho+\overline{b}|\xi|$. 
With $|\tau|\ge |\xi|^\varrho+\overline{b}|\xi|$ and $|\xi|\ge1$, the left function in the integrand 
\begin{equation*}
\frac{\big(1+|\xi|^{2-2\varrho}\big)|\tau|^{2\varsigma}}{|\tau|^2+\overline{b}^2|\xi|^2}
\le \frac{|\xi|^{4-4\varrho}+2|\tau|^{4\varsigma}}{|\tau|^2+\overline{b}^2|\xi|^2}
\le C_b\left(|\xi|^{2-4\varrho} +|\tau|^{4\varsigma-2}\right)\le 2C_b. 
\end{equation*}
It turns out that 
\begin{equation*}
\begin{split}
H_4
&\le C_b\|\psi\|_{W^{1,\infty}}^2\int G(\tau,\xi)
\left(1+\frac{|\xi|^{1-\alpha}|\xi|^{1-\alpha}|\xi|^{2\varsigma}}{(|\xi|^{1+\varrho})^{2-\alpha}} + \frac{|\xi||\xi|^{3-2\varrho}|\xi|^{2\varsigma}}{(|\xi|^{1+\varrho})^2}
\right) \mathbbm{1}_{|\xi|\ge 1}\dif\tau\dif\xi \\
&\le C_b\|\psi\|_{W^{1,\infty}}^2\left(\| h\|_{L^2}^2+\| g_0\|_{L^2}^2+\| g_1\|_{L^2}^2\right).
\end{split}
\end{equation*}
Similarly, we can treat the term $H_5$ and hence bound the quantity $\| h_\psi\|^2_{\dot{H}^{\varsigma}}$. 
This completes the proof. 
\end{proof}

\subsection{Hypoelliptic transfer of regularity} 
With the help of the above velocity averaging lemma, the hypoelliptic methodology given in \cite{Bo} allows us to see that the extra regularity of the solution in velocity can be transferred to time-space variables. 

\begin{lemma}\label{VAL}
Under the same assumption as in Lemma~\ref{VA}, in addition provided that  
\begin{equation}\label{VAderivative v}
h\in L^2_{t,x}\left(H^\delta_v(B_R)\right),
\end{equation}
for some $\delta\in(0,1]$, then for $\gamma:=\frac{2\alpha\delta}{(4+\alpha)(2+2\delta+d_2)}$ and for any ball $B_r$ with $r\in(0,R)$, we have
\begin{equation*}\label{VACor}
\| h\|_{{H}^\gamma(\mathbb{R}\times\mathbb{R}^{d_1}\times B_r)}\le
C_b\left(\| h\|_{L^2_{t,x}(H_v^\delta)}+\| g_1\|_{L^2}+\| g_0\|_{L^2}\right), 
\end{equation*}
where the constant $C_b>0$ depends on $d_1$, $d_2$, $K$, $R-r$, $\| b\|_{L^\infty}$, and $\| Db\|_{L^2}$. 
\end{lemma}

\begin{remark}\label{rkVAL}
The regularizing effect produced in the lemma is not optimal. 
For instance, by adjusting the parameters in the proof below with $\epsilon(\tau,\xi)=[\epsilon_0(|\tau|+|\xi|)]^{-\frac{\gamma}{\delta}}$ for some small $\epsilon_0>0$ 
and $m(\xi)$ given by $m=1$ for $|\xi|\le1$, $m=|\xi|^\varrho$ for $|\xi|>1$, where $\varrho:=\frac{2\alpha\delta+2d_2+\alpha d_2}{2\alpha\delta+4d_2+\alpha d_2}$, 
it turns out that $h\in H^{\gamma}$ with $\gamma=\frac{2\alpha\delta}{2\alpha\delta+4d_2+\alpha d_2}$. 
Moreover, as in Remark \ref{rklip}, if $\|b\|_{W^{1,\infty}}$ is bounded, we may pick $\varrho:=\frac{\alpha\delta+d_2}{\alpha\delta+2d_2}$ here to deduce that $h\in H^{\gamma}$ with $\gamma=\frac{\alpha\delta}{\alpha\delta+2d_2}$. 
We are not going to show the detailed computation because of its tediousness. 
\end{remark}

\begin{remark}\label{rkNA}
Consider the case that $b(v)\in W^{1,\infty}(B_R,\mathbb{R}^{d_1})$ satisfies the nondegenerate condition \eqref{nond} on any rescaled and translated balls, that is, for any $B_r(v_0)$ with $\overline{B}_{r}(v_0)\subset B_R$, 
\begin{equation}\label{NA}
\forall \mu\in\mathbb{R},\ \ \forall \nu\in\mathbb{S}^{{d_1}-1}, \ \forall \epsilon>0,\ \  
\big|\{ v\in B_r(v_0): | \mu+b(v)\cdot\nu|\le \epsilon\}\big| 
\lesssim r^{{d_2}-1}\epsilon^\alpha. 
\end{equation}
Then, for any fixed cut-off function $\zeta\in C^\infty_c(\mathbb{R})$ and fixed nonnegative function $\rho\in C^\infty_c(B_1)$, 
\begin{align*}
\int_{B_\epsilon(v)} \zeta^2\left(\frac{\tau+b(w)\cdot\xi}{m}\right) \frac{1}{\epsilon^{d_2}}\rho\left(\frac{v-w}{\epsilon}\right)\dif w
\lesssim \frac{1}{\epsilon}\left(\frac{m}{|\xi|}\right)^\alpha\mathbbm{1}_{|\tau|\le m+|\xi|\|b\|_{L^\infty}}. 
\end{align*}
In view of this type of estimate (in contrast to \eqref{J1}, \eqref{J2}), by picking $\epsilon(\tau,\xi)=[\epsilon_0(|\tau|+|\xi|)]^{-\frac{\gamma}{\delta}}$ for some small $\epsilon_0>0$ and picking $m=1$ for $|\xi|\le1$, $m=|\xi|^\frac{1+\alpha\delta}{2+\alpha\delta}$ for $|\xi|>1$ in the following proof, 
we will derive $h\in H^{\gamma}$ with $\gamma=\frac{\alpha\delta}{2+\alpha\delta}$. 
Noticing that setting $\alpha=1$ in \eqref{NA} gives the nondegenerate condition \eqref{N}, we then have $\gamma=\frac{\delta}{2+\delta}$. 
In particular, this coincides with the regularity result given in \cite{Bo} in the case of $b(v)=v$. 
\end{remark}

Let us turn to the proof of the lemma. 
\begin{proof}
For the sake of clarity, we assume $h$, $g_1$, and $g_0$ are compactly supported in $v\in B_R$ by a simple localization argument (multiply the solution with a cut-off function). 
We remark that if they are already compactly supported with respect to $v$, then the constant $C_b$ in the conclusion above will not depend on $R-r$. 

Introduce a sequence of radial functions $\{\rho_\epsilon(v)\}_{\epsilon>0}\subset C^\infty_c(\mathbb{R}^{d_2})$ such that
\begin{equation*}
 \rho_1\ge 0,\quad 
 {\rm supp\ }\rho_1\subset B_1,\quad 
 \int_{\mathbb{R}^{d_2}}\rho_1(v)\dif v=1 \quad {\rm and}\quad 
\rho_\epsilon(v)=\frac{1}{\epsilon^{d_2}}\rho_1\left(\frac{v}{\epsilon}\right).
\end{equation*}
Write 
\begin{equation}\label{VAsplit}
h(t,x,v)=\left(h-h*_v\rho_\epsilon\right)(t,x,v)+h*_v\rho_\epsilon(t,x,v). 
\end{equation}

We will see that in the averaging (with respect to $v$) sense, one can extract a factor $\epsilon$ to some positive power from the first term on the right-hand side of \eqref{VAsplit} due to our assumption \eqref{VAderivative v}. 
Meanwhile, a factor $\epsilon$ to some negative power can be extracted from the second term by following the demonstration of the previous lemma with slight adjustment in the argument. 
Afterward, the regularity of the solution $h$ itself will be obtained by identifying the parameter $\epsilon=\epsilon(\tau,\xi)$ which is defined in the Fourier space of $t,x$. 

By performing the Fourier transforms $\mathscr{F}$ in all variables $(t,x,v)\rightarrow(\tau,\xi,\eta)$ and $\hat{}$ with respect to $(t,x)\rightarrow(\tau,\xi)$, we have 
\begin{equation*}
\|h\|_{L^2_v\dot{H}^{\gamma}_{t,x}}
\le \big\|(|\tau|+|\xi|)^{\gamma}|\mathscr{F}\left(h-h*_v\rho_{\epsilon(\tau,\xi)}\right)|\big\|_{L^2_{\tau,\xi,\eta}}
+\big\|(|\tau|+|\xi|)^\gamma \big|\hat{h}*_v\rho_{\epsilon(\tau,\xi)}\big|\big\|_{L^2_{\tau,\xi,v}}. 
\end{equation*}

First, let us treat the term about $h-h*_v\rho_\epsilon$. We denote $\mathscr{F}_v$ the Fourier transform with respect to $v\rightarrow\eta$. It turns out that
\begin{align*}
|\mathscr{F}\left(h-h*_v\rho_\epsilon\right)|&= |1-(\mathscr{F}_v\rho_1)(\epsilon\eta)||\mathscr{F}h|\\
&\lesssim |\eta|^\delta|\mathscr{F}h| \sup_{|\eta|\le\epsilon^{-1}}\left(|\eta|^{-\delta}|\epsilon\eta|\right) +
 |\eta|^\delta|\mathscr{F}h| \sup_{|\eta|>\epsilon^{-1}}|\eta|^{-\delta}
=2\epsilon^\delta|\eta|^\delta|\mathscr{F}h|, 
\end{align*}
where we used the facts that $\mathscr{F}_v\rho_1(0)=1$ and $\mathscr{F}_v\rho_1(\eta)$, as well as its derivatives, is bounded in $\mathbb{R}_\eta^{d_2}$. 
Then, with $\gamma>0$ to be determined, selecting the parameter $\epsilon=(|\tau|+|\xi|)^{-\frac{\gamma}{\delta}}$ and using the assumption \eqref{VAderivative v} yield
\begin{align}\label{VAs1}
\big\|(|\tau|+|\xi|)^{\gamma}|\mathscr{F}\left(h-h*_v\rho_\epsilon\right)|\big\|_{L^2_{\tau,\xi,\eta}}
\lesssim \|h\|_{L_{t,x}^2\dot{H}_v^\delta}. 
\end{align}

Next, for the term about $h*_v\rho_\epsilon$, we write 
\begin{equation*}
\big\|(|\tau|+|\xi|)^\gamma \big|\hat{h}*_v\rho_\epsilon\big|\big\|_{L^2_{\tau,\xi,v}}^2
= \int \big| J(\tau,\xi,v) \big|^2\left(|\tau|+|\xi|\right)^{2\gamma}  \dif\tau\dif\xi\dif v,
\end{equation*}
where we set
\begin{equation*}
J(\tau,\xi,v):=\int \hat{h}(\tau,\xi,w)\rho_\epsilon(v-w) \dif w. 
\end{equation*}
Consider the decomposition $J(\tau,\xi,v)=J_1(\tau,\xi,v)+J_2(\tau,\xi,v)$ as follows. With a cut-off function $\zeta$ supported in $[-1,1]$ valued in $[0,1]$ such that $\zeta|_{[-\frac{1}{2},\frac{1}{2}]}\equiv 1$, 
as well as taking $m(\xi)$ such that $m=1$ for $|\xi|\le1$ and $m=|\xi|^{\frac{2+\alpha}{4+\alpha}}$ for $|\xi|>1$ as before, we set
\begin{equation*}
J_1(\tau,\xi,v):=\int \hat{h}(\tau,\xi,w)\rho_\epsilon(v-w)\zeta\left(\frac{\tau+b(w)\cdot\xi}{m}\right)\dif w,
\end{equation*} 
By the relation \eqref{VArelation}, 
\begin{equation*}
J_2(\tau,\xi,v)
={\rm i}\int
\hat{g}_1(\tau,\xi,w)\cdot \nabla_v \left[\Phi(\tau,\xi,w)\rho_\epsilon(v-w)\right]
-\hat{g}_0(\tau,\xi,w)\Phi(\tau,\xi,w)\rho_\epsilon(v-w)\dif w,
\end{equation*} 	
where $\Phi$ is provided by \eqref{w(v)}. In such a setting, we only need to replace the estimates of $|I_1|$ and $|I_2|$ as shown in \eqref{I1} and \eqref{I12} in the previous lemma by the estimates of $\| J_1\|_{L^2_v}$ and $\| J_2\|_{L^2_v}$. 
More precisely, 
\begin{align}\label{J1}
\int |J_1(\tau,\xi,v)|^2\dif v
&\le\int  \big\|\hat{h}(\tau,\xi,\cdot)\big\|_{L^2}^2
\bigg\|\zeta\left(\frac{\tau+b(\cdot)\cdot\xi}{m}\right)\rho_\epsilon(v-\cdot)\bigg\|_{L^2}^2\dif v  \nonumber\\
&\lesssim\frac{1}{\epsilon^{d_2}}\|\hat{h}(\tau,\xi,v)\|_{L^2_v}^2\bigg\|\zeta\left(\frac{\tau+b(v)\cdot\xi}{m}\right)\bigg\|_{L^2_v}^2.
\end{align}
In like manner, 
\begin{align}\label{J2}
\int |J_2(\tau,\xi,v)|^2\dif v
&\le \left(\|\hat{g}_1\|_{L^2_v}^2+\|\hat{g}_0\|_{L^2_v}^2\right)
\left(\|\nabla\rho_\epsilon\|_{L^2}\|\Phi\|_{L^2_v}^2+ \|\rho_\epsilon\|_{L^2}\|\nabla_v\Phi\|_{L^2_v}^2\right) \nonumber\\
&\lesssim \frac{1+\epsilon^2}{\epsilon^{d_2+2}}\left(\|\hat{g}_1\|_{L^2_v}^2+\|\hat{g}_0\|_{L^2_v}^2\right)
\left(\|\Phi\|_{L^2_v}^2+\|\nabla_v\Phi\|_{L^2_v}^2\right). 
\end{align}

Combining \eqref{J1} with \eqref{J2}, as well as estimating the terms $\|\zeta\left((\tau+b(v)\cdot\xi)/m\right)\|_{L^2_v}$, $\|\Phi\|_{L^2_v}^2+\|\nabla_v\Phi\|_{L^2_v}^2$ and performing the subsequent operations as before, we conclude that 
\begin{align}\label{VAs2}
\big\|(|\tau|+|\xi|)^\gamma \big|\hat{h}*_v\rho_\epsilon\big|\big\|_{L^2_{\tau,\xi,v}}^2
=& \int |J(\tau,\xi,v)|^2 \left(|\tau|+|\xi|\right)^{\frac{2\alpha}{4+\alpha}} \left(|\tau|+|\xi|\right)^{2\gamma-\frac{2\alpha}{4+\alpha}} \dif\tau\dif\xi \nonumber\\
\lesssim& \left(\| h\|_{L^2}^2+\| g_1\|_{L^2}^2+\| g_0\|_{L^2}^2\right)
\sup_{|\tau|+|\xi|\le1}\left(|\tau|+|\xi|\right)^{2\gamma-\frac{2\alpha}{4+\alpha} +\frac{d_2\gamma}{\delta}}\nonumber\\
&+\left(\| h\|_{L^2}^2+\| g_1\|_{L^2}^2+\| g_0\|_{L^2}^2\right)
\sup_{|\tau|+|\xi|>1}\left(|\tau|+|\xi|\right)^{2\gamma-\frac{2\alpha}{4+\alpha}+\frac{(d_2+2)\gamma}{\delta}} \nonumber\\
\lesssim& \| h\|_{L^2}^2+\| g_1\|_{L^2}^2+\| g_0\|_{L^2}^2, 
\end{align}
where we choose $\gamma$ such that $2\gamma-\frac{2\alpha}{4+\alpha}+\frac{(d_2+2)\gamma}{\delta}=0$. 

From \eqref{VAs1}, \eqref{VAs2}, and the condition \eqref{VAderivative v}, we derive the desired result. 
\end{proof}

\section{Local boundedness}\label{sectionbdd}
Both De Giorgi's approach and Moser's approach toward Theorem~\ref{bdd} are presented in this section. 

They both begin with exhibiting a gain of integrability for solutions. To reach an $L^2$-$L^\infty$ estimate, De Giorgi did the estimate with a monotone truncation to obtain a recurrence relation, whereas Moser intended to construct an integrability-gaining iteration procedure directly (see a more precise explanation below). 
By comparing with each other, one is able to see that the ideas involved are essentially the same. 
By contrast with \cite{GIMV}, we can also apply Moser's method to treat general bounded solutions to the equation with a source term. 

In view of the weak formulation of \eqref{FP}, a simple computation and approximation procedure yields the following fact. 
\begin{lemma}\label{subsol}
	If $f$ is a subsolution to \eqref{FP} in $Q_1$, 
	then for any nonnegative convex function $\chi\in W^{1,\infty}(\mathbb{R})$ with $\chi'\ge0$, $\chi(f)$ is also a subsolution to the same equation with the souce term ${s}$ replaced by ${s}\chi'(f)$. 
\end{lemma}
In particular, for any subsolution $f$ of \eqref{FP}, if $\chi(f)=f^+$, then $\chi'(f)=\mathbbm{1}_{f\ge 0}$ and $\chi(f)$ is also a subsolution to the same equation with ${s}$ replaced by ${s}\mathbbm{1}_{f\ge 0}$.

\subsection{Moser's approach}
In the elliptic setting, Moser's strategy runs as follows. As we choose the test function to be a cut-off function multiplied with a power of the solution, we can obtain an $L^2$-energy estimate for a power of the solution, that is, an estimate of the $L^2$-norm of the derivative of a power of the solution in term of the $L^2$-norm of the same power of the solution itself. 
Then, the Sobolev inequality implies an estimate of the $L^{p_2}$-norm of the solution in term of its $L^{p_1}$-norm with ${p_2}> {p_1}$, that is, a reversed H\"{o}lder's inequality. With a gain of higher and higher $L^{p_n}$-integrability and a careful choice of the cut-off functions, $p_n$ goes to infinity and boundedness is reached. 

However, standard elliptic equation techniques (see \cite{GT} and \cite{HL}, for instance) only enable us to achieve the $L^2$-energy estimate with respect to the velocity variable $v$. 
Instead, the hypoelliptic structure of \eqref{FP} will also allow us to prove $H^\gamma$-bound on a barrier function that dominates solutions of \eqref{FP} for some $\gamma\in(0,1)$. 
The hypoellipticity comes from the ingredients of both elliptic estimate and velocity averaging. 

Before turning to the proof of Theorem~\ref{bdd} following Moser's path, we introduce a simple but useful lemma \cite[Lemma~4.3]{HL} first. 
\begin{lemma}\label{technicallemma1}
Let $\psi(r)$ be a nonnegative bounded function on $0\le r_0\le r\le r_1$. Suppose that for $r_0\le r<s\le r_1$ we have 
$$\psi(r)\le \epsilon \psi(s)+c(s-r)^{-\iota}$$
for some $\epsilon\in[0,1)$, $\iota>0$, and $c\ge 0$. Then, there exists a positive constant $C$ depending only on $\epsilon$ and $\iota$ such that, for any $r_0\le r<s\le r_1$, we have
$$\psi(r)\le Cc(s-r)^{-\iota}.$$
\end{lemma}

We are now in a position to prove the local boundedness theorem. 

\begin{proof}[Proof of Theorem~\ref{bdd} (Moser's approach)]
The proof will proceed in four steps. 
	
\noindent{\textbf{Step 1.}} We will establish the local energy estimate with respect to the variable $v$; see \eqref{energy} below. 
Let $0<r<R\le 1$ and $\overline{r}:=\frac{r+R}{2}$. With $l>0$, we set $f_{l}:=f^++l$. Take a cut-off function $\phi$ satisfying $0\le\phi\le 1$, $\phi|_{Q_{\overline{r}}}\equiv 1$, ${\rm supp}\phi\subset Q_R$, as well as the following properties: 
\begin{equation}\label{enertest}
|\partial_t\phi|\le \frac{C}{(R-r)^2}, \quad
|\nabla_x\phi|\le \frac{C}{(R-r)^3}{\quad\rm and\quad}
|\nabla_v\phi|\le \frac{C}{R-r} {\quad\rm in\ } Q_1. 
\end{equation}

For clarity of purpose, we deal with bounded solutions here so that the test function of the form $\varphi=\phi^2(f_{l}^\beta-l^\beta)$ with $\beta\ge1$ is admissible; generally, one may consider $\varphi_M=\phi^2(f_{l,M}^{\beta-1}f_l-l^\beta)$, with $f_{l,M}=f_l$ if $f_l\le M$, $f_{l,M}=M$ if $f_l>M$, and send $M\rightarrow\infty$ in the end. 
Then, 

\begin{equation}\label{ener0}
\int_{Q_R}\varphi(\partial_t+b\cdot\nabla_x)f\le-\int_{Q_R}A\nabla_v\varphi\cdot\nabla_vf+\int_{Q_R}\varphi B\cdot\nabla_vf+\int_{Q_R}\varphi s. 
\end{equation}

Using the elliptic condition in \eqref{H}, as well as applying the Cauchy-Schwarz inequality, yields
\begin{equation}\label{ener1}
\begin{split}
\int_{Q_R} A\nabla_v\varphi\cdot\nabla_vf
&\ge  \beta\int_{Q_R} \phi^2f_{l}^{\beta-1}A\nabla_v f_{l}\cdot\nabla_vf_{l}+2\int_{Q_R}\phi (f_{l}^\beta-l^\beta) A\nabla_v\phi \cdot\nabla_vf_{l}\\
&\ge  \lambda\beta\int_{Q_R} \phi^2f_{l}^{\beta-1}|\nabla_vf_{l}|^2-2\Lambda\int_{Q_R}\phi f_{l}^\beta|\nabla_v\phi||\nabla_vf_{l}|\\
&\ge  \lambda\left(\beta-\frac{1}{2}\right)\int_{Q_R} \phi^2f_{l}^{\beta-1}|\nabla_vf_{l}|^2
-\frac{2\Lambda^2}{\lambda}\int_{Q_R}f_{l}^{\beta+1}|\nabla_v\phi|^2,  
\end{split}
\end{equation}
where we used the facts that $\nabla_vf_{l}=\nabla_vf$ a.e. in $\{f>0\}$ and $\varphi=0$, $\nabla_vf_{l}=0$ a.e. in $\{f\le 0\}$ so that all the integrals above are effective only over the set $\{f>0\}$. For the same reason and after integrating by parts, we have
\begin{equation}\label{ener2}
\int_{Q_R}\varphi(\partial_t+b\cdot\nabla_x)f
= \int_{Q_R} \phi^2(f_{l}^\beta-l^\beta)(\partial_t+b\cdot\nabla_x)f_{l}
\ge -4\int_{Q_R} \phi f_{l}^{\beta+1}|(\partial_t+b\cdot\nabla_x)\phi|.
\end{equation}
In addition, 
\begin{equation}\label{ener3}
\int_{Q_R}\varphi B\cdot\nabla_vf
\le  \Lambda\int_{Q_R}\phi^2f_{l}^\beta|\nabla_vf_{l}|
\le  \frac{\lambda}{4}\int_{Q_R} \phi^2f_{l}^{\beta-1}|\nabla_vf_{l}|^2
+\frac{\Lambda^2}{\lambda}\int_{Q_R}\phi^2 f_{l}^{\beta+1}. 
\end{equation}

Set $u:=f_{l}^\frac{\beta+1}{2}$ and $c_s:=\frac{|s|}{l}$ with $l:=\| s\|_{L^q(Q_1)}$. It turns out that 
\begin{equation*}
|\nabla_v u|^2\le \beta^2 f_{l}^{\beta-1}|\nabla_v f_{l}|^2\quad{\rm and}\quad \frac{|s|}{f_{l}}\le c_s.  
\end{equation*}
Therefore, combining this with \eqref{ener0}, \eqref{ener1}, \eqref{ener2}, and \eqref{ener3}, we obtain 
\begin{equation}\label{energy-grad}
\int_{Q_R}|\nabla_v(\phi u)|^2
\le C\beta\left(1+\|\partial_t\phi\|_{L^\infty}+\|\nabla_x\phi\|_{L^\infty}
+\|\nabla_v\phi\|_{L^\infty}^2\right) \int_{Q_R} u^2+C\beta\int_{Q_R}c_s\phi^2u^2. 
\end{equation}
Observing that $\| c_s\|_{L^q}\le 1$, applying H\"older's inequality and an interpolation inequality, we have
\begin{equation}\label{energy-sourse}
\begin{split}
\beta\int_{Q_R} c_s\phi^2u^2
&\le \beta\| c_s\|_{L^q}\| \phi u\|_{L^{2\kappa}}^{2\theta}\|\phi u\|_{L^2}^{2-2\theta}\\
&\le 
C(\epsilon,\theta)(R-r)^{-\frac{2\theta}{1-\theta}}\beta^\frac{1}{1-\theta}\|\phi u\|_{L^2}^2+\epsilon(R-r)^2\|\phi u\|_{L^{2\kappa}}^2, 
\end{split}
\end{equation}
where $\epsilon>0$, $\kappa\ge 1$ will be determined and $\theta=\frac{\kappa}{(\kappa-1)q}\in (0,1)$ which requires $q>\frac{\kappa}{\kappa-1}$. In this case, $C(\epsilon,\theta)$ only depends on $\epsilon$, $\kappa$ and universal constants, so we rewrite it as $C_{\epsilon,\kappa}$. 

On account of our choice of the function $\phi$ satisfying \eqref{enertest} with $0<r<R\le1$, \eqref{energy-grad} and \eqref{energy-sourse} imply that
\begin{equation}\label{energy}
\|\nabla_vu\|^2_{L^2(Q_{\overline{r}})}
\le \frac{C_{\epsilon,\kappa}\beta^\vartheta}{(R-r)^{3\vartheta}}\|u\|^2_{L^2(Q_R)}
+C\epsilon(R-r)^2\| u\|_{L^{2\kappa}(Q_R)}^2,
\end{equation}
where $\vartheta\ge 1$ is given by $\vartheta:=\frac{1}{1-\theta}$.

\medskip\noindent{\textbf{Step 2.}} Estimates of source terms and comparison. \\
In view of the energy estimate \eqref{energy} and the Sobolev inequality, we are able to gain  higher integrability of $f_{l}$ with respect to the variable $v$. Now, we are going to transfer the integrability to the variables $t$ and $x$, thanks to velocity averaging. 
Introduce another cut-off function 
$\eta$ such that $0\le\eta\le 1$, $\eta|_{Q_r}\equiv 1$, ${\rm supp}\eta\subset Q_{\overline{r}}$, and
\begin{equation*}\label{enertest0}
|\partial_t\eta|\le \frac{C}{(R-r)^2}, \quad
|\nabla_x\eta|\le \frac{C}{(R-r)^3}{,\quad\rm and\quad}
|\nabla_v\eta|\le \frac{C}{R-r} {\quad\rm in\ } Q_1. 
\end{equation*}
Then, the function $\eta u$ is a nonnegative subsolution to an equation of the type 
\begin{equation}\label{whole}
\left(\partial_t+b\cdot\nabla_x\right)h
= {\rm div}_v\left(A\nabla_v h\right)+{\rm div}_v g_1+g_0 {\quad\rm in\ } \mathbb{R}^{1+{d_1}+{d_2}},
\end{equation}
where $h(t,x,v)$ is unknown and $g_1$, $g_0$ are given by 
\begin{equation*} 
\left\{ 
\begin{aligned}
\ &g_1=-Au\nabla_v\eta,  \\
\ &g_0=(B\eta-A\nabla_v\eta)\cdot\nabla_vu+u(\partial_t+b\cdot\nabla_x)\eta +\beta c_s\eta u, 
\end{aligned}
\right. 
\end{equation*}
as we observe that $f\eta$, $g_1$, and $g_0$ are supported in $Q_{\overline{r}}$. 
Besides, we can estimate $g_1$ and $g_0$ by the local energy estimate \eqref{energy}. Indeed, we have
\begin{align}\label{g1g0} 
\| g_1\|^2_{L^2({Q_{\overline{r}}})}+\| g_0\|^2_{L^2({Q_{\overline{r}}})}
\le& C\left(\|\partial_t\eta\|_{L^\infty}^2+\|\nabla_x\eta\|^2_{L^\infty}+\|\nabla_v\eta\|_{L^\infty}^2\right) \| u\|^2_{L^2(Q_R)}\nonumber\\ 
&+C\left(1+\|\nabla_v\eta\|_{L^\infty}^2\right)\|\nabla_vu\|^2_{L^2(Q_{\overline{r}})} + \| \beta c_su \|^2_{L^2(Q_R)}\nonumber\\
\le& \frac{C}{(R-r)^6}\| u\|^2_{L^2(Q_R)}\nonumber\\
&+\frac{C}{(R-r)^2}\|\nabla_vu\|^2_{L^2(Q_{\overline{r}})}
+C_{\epsilon,\kappa}\beta^{2\vartheta}\| u\|_{L^2(Q_R)}+\epsilon\| u\|_{L^{2\kappa}(Q_R)}^2\nonumber\\
\le& \frac{C_{\epsilon,\kappa}\beta^{2\vartheta}}{(R-r)^{6\vartheta}}\| u\|^2_{L^2(Q_R)}
+C\epsilon\| u\|_{L^{2\kappa}(Q_R)}^2,
\end{align}
where we used the fact that $\beta,\vartheta\ge1$ and treated the term $\| \beta c_su \|^2_{L^2(Q_R)}$ similarly to \eqref{energy-sourse},  
\begin{equation*}
\| \beta c_su\|_{L^2(Q_R)}^2
\le \beta^2\| c_s\|_{L^q(Q_R)}^2\| u\|_{L^{2\kappa}(Q_R)}^{2\theta}\| u\|_{L^2(Q_R)}^{2-2\theta}
\le C_{\epsilon,\kappa}\beta^{\frac{2}{1-\theta}}\| u\|_{L^2(Q_R)}
+\epsilon\| u\|_{L^{2\kappa}(Q_R)}^2, 
\end{equation*}
and this requires 
\begin{equation}\label{bddindex}
q>\frac{2\kappa}{\kappa-1}.
\end{equation} 

Moreover, the solution $h$ to \eqref{whole} with the initial boundary conditions 
\begin{equation*} 
\left\{ 
\begin{aligned}
\ & h(t,x,v)=0 {\rm\ if\ } |v|={\overline{r}} {, \rm\ or\ } |x|={\overline{r}}^2, 
 {\rm\ and\ } b(v)\cdot x<0
 \\
\ & h(t_0-{\overline{r}}^2,x,v)=0
\end{aligned}
\right. 
\end{equation*}
is supported in $Q_{\overline{r}}$. By the maximum principle, we have 
\begin{equation}\label{comparison}
0\le \eta u\le h {\quad\rm in\ } Q_{\overline{r}}.
\end{equation} 

\medskip\noindent{\textbf{Step 3.}} The gain of integrability. \\
We are in a position to focus on \eqref{whole}. Since the solution $h$ is supported in $Q_{\overline{r}}$, 
 for any $T\in(-\overline{r}^2,0]$,
integrating \eqref{whole} against $h\mathbbm{1}_{t\le T}$ yields
\begin{equation*}\label{globalenergy1}
\begin{split}
\frac{1}{2}\int_{Q_{\overline{r}}}\partial_t(h^2)\mathbbm{1}_{t\le T}
\le&-\lambda\int_{Q_{\overline{r}}}|\nabla_vh|^2\mathbbm{1}_{t\le T}
-\int_{Q_{\overline{r}}}\left(g_1\cdot\nabla_vh-g_0h\right)\mathbbm{1}_{t\le T} \\
\le&-\frac{\lambda}{2}\int_{Q_{\overline{r}}}|\nabla_vh|^2\mathbbm{1}_{t\le T}
+C\| g_1\|_{L^2({Q_{\overline{r}}})}^2
+\| g_0\|_{L^2({Q_{\overline{r}}})}\| h\|_{L^2({Q_{\overline{r}}})}, 
\end{split}
\end{equation*} 
so that we are able to write the global energy estimate with respect to the velocity variable for the solution $h$,  
\begin{equation*}\label{globalenergy2}
\begin{split}
\sup_{t\in(-{\overline{r}}^2,0)}\int_{Q^t_{\overline{r}}}h^2+\lambda\int_{Q_{\overline{r}}}|\nabla_vh|^2
&\le C\| g_1\|^2_{L^2({Q_{\overline{r}}})}+C\| g_0\|_{L^2({Q_{\overline{r}}})}\left(\sup_{t\in((-{\overline{r}}^2,0)}\int_{Q^t_{\overline{r}}}h^2\right)^\frac{1}{2}\\
&\le C\left(\| g_1\|^2_{L^2({Q_{\overline{r}}})}+ \| g_0\|^2_{L^2({Q_{\overline{r}}})}\right)+\frac{1}{2}\sup_{t\in((-{\overline{r}}^2,0)}\int_{Q^t_{\overline{r}}}h^2,
\end{split}
\end{equation*} 
where we denote $Q_{\overline{r}}^t:=\{(x,v):(t,x,v)\in Q_{\overline{r}}\}$. 
In particular, recalling that $0<r<R\le 1$ and $\overline{r}=\frac{r+R}{2}$, we get
\begin{equation*}
\| h\|^2_{L^2_{t,x}H_v^1({Q_{\overline{r}}})}\lesssim \| g_1\|^2_{L^2({Q_{\overline{r}}})}+\| g_0\|^2_{L^2({Q_{\overline{r}}})}. 
\end{equation*}
Then, applying Lemma~\ref{VAL} (with $\delta=1$ and $\gamma=\frac{2\alpha}{(4+\alpha)(4+d_2)}$) to \eqref{whole} with $h$ supported in $Q_{\overline{r}}$, we have 
\begin{equation*}
\| h\|_{H^\gamma(Q_r)}
\lesssim \| h\|_{L_{t,x}^2H_v^1({Q_{\overline{r}}})}+\| g_1\|_{L^2({Q_{\overline{r}}})}+\| g_0\|_{L^2({Q_{\overline{r}}})}
\lesssim \| g_1\|_{L^2({Q_{\overline{r}}})}+\| g_0\|_{L^2({Q_{\overline{r}}})}. 
\end{equation*}
Combining this with the estimate \eqref{g1g0} for $\| g_1\|_{L^2}+\| g_0\|_{L^2}$ yields
\begin{equation*}
\| h\|^2_{H^\gamma(Q_r)}
\le \frac{C_{\epsilon,\kappa}\beta^{2\vartheta}}{(R-r)^{6\vartheta}}\| u\|^2_{L^2(Q_R)}
+C\epsilon\| u\|_{L^{2\kappa}(Q_R)}^2. 
\end{equation*}
Then, the Sobolev inequality $\| h\|_{L^{2\kappa}(Q_r)}\lesssim\| h\|_{H^\gamma(Q_r)}$ and the fact \eqref{comparison} imply higher integrability of $f_{l}$. 
To be more specific, recalling that $$u=f_{l}^{\frac{p}{2}} {\quad\rm with\quad} p:=\beta+1\ge 2,$$ and rewriting the result in terms of $f_{l}$ gives 
\begin{equation*}
\| f_{l}\|_{L^{\kappa p}(Q_r)}\le \| h\|_{L^{2\kappa}(Q_r)}^{\frac{2}{p}}
\le \left(\frac{C}{(R-r)^{6\vartheta}}\right)^{\frac{1}{p}}\| f_{l}\|_{L^p(Q_R)}
+\frac{1}{2}\| f_{l}\|_{L^{\kappa p}(Q_R)}, 
\end{equation*}
where we set the positive constant $\epsilon$ to be small enough and $\kappa>1$ is given by the Sobolev conjugate 
\begin{equation}\label{sobolevconjugate}
\frac{1}{2\kappa}=\frac{1}{2}-\frac{\gamma}{1+{d_1}+{d_2}}. 
\end{equation} 
Thanks to Lemma~\ref{technicallemma1}, we achieve
\begin{equation*}
\| f_{l}\|_{L^{\kappa p}(Q_r)}
\le \left(\frac{C}{(R-r)^{6\vartheta}}\right)^{\frac{1}{p}}\| f_{l}\|_{L^p(Q_R)}.
\end{equation*}

\medskip\noindent{\textbf{Step 4.}} The iteration. \\
We finally construct an iteration by taking $r_n=\frac{1}{2}+\frac{1}{2^{n+1}}$ and $p_n=2\kappa^n$ with $n\in\mathbb{N}$ so that 
\begin{equation*}
\| f_{l}\|_{L^{p_{n+1}}(Q_{r_{n+1}})}
\le C^{\frac{n}{p_n}}\| f_{l}\|_{L^{p_n}(Q_{r_n})}
\le C^{\sum_{k\in\mathbb{N}}\frac{k}{p_k}}\| f_{l}\|_{L^{p_0}(Q_{r_0})}. 
\end{equation*}
Note that the series $\sum_{k\in\mathbb{N}}\frac{k}{p_k}$ converges. Sending $n\rightarrow\infty$ and recalling that $l=\| s\|_{L^q(Q_1)}$, we derive 
\begin{equation*}
\| f^+\|_{L^\infty(Q_\frac{1}{2})}
\le C\left(\| f^+\|_{L^{2}(Q_1)}+\| s\|_{L^{q}(Q_1)}\right).
\end{equation*}
The proof now is complete. 
\end{proof}

\subsection{De Giorgi's approach}
\begin{proof}[De Giorgi's approach to Theorem~\ref{bdd}]
	We will pick two positive sequences $k_n\rightarrow k_\infty$, $r_n\rightarrow r_\infty$ with $0<k_\infty,r_\infty<\infty$ and establish that $$\mathcal{A}_n:=\|(f-k_n)^+\|_{L^2(Q_{r_n})}\rightarrow0$$
	by proving a type of inequality $$\mathcal{A}_n\le C^n\mathcal{A}_{n-1}^{1+\epsilon}$$
	for some positive universal constant $\epsilon$; see \eqref{degiorgiinduction0} below for the precise expression. 
	The proof is split into four steps. 
	
\medskip\noindent{\textbf{Step 1.}}	 The energy estimate. \\ 
Let $0<r<R\le1$ and $\overline{r}:=\frac{r+R}{2}$. 
As before, we consider again a cut-off function $\phi$ supported in $Q_R$, valued in $[0,1]$, $\phi|_{Q_{\overline{r}}}\equiv 1$, and
\begin{equation*}
|\partial_t\phi|\le \frac{C}{(R-r)^2}, \quad
|\nabla_x\phi|\le \frac{C}{(R-r)^3}, \quad
|\nabla_v\phi|\le \frac{C}{R-r} {\quad\rm in\ } Q_1. 
\end{equation*}
We will write the energy estimate by choosing the test function $\varphi:=\phi^2f_{k}$, where we define $f_k:=(f-k)^+$ for $k>0$.
	
Similarly to \eqref{ener1}, \eqref{ener2}, and \eqref{ener3}, it easily turns out that
	\begin{equation*}
	\begin{split}
	&\int_{Q_R} A\nabla_v\varphi\cdot\nabla_vf \ge\frac{\lambda}{2}\int_{Q_R}\phi^2|\nabla_vf_{k}|^2-\frac{2\Lambda^2}{\lambda}\int_{Q_R}f_{k}^2|\nabla_v\phi|^2, \\
	&\int_{Q_R} \varphi(\partial_t+b\cdot\nabla_x)f \ge-\int_{Q_R} \phi f_{k}^2|(\partial_t+b\cdot\nabla_x)\phi|,\\
	&\int_{Q_R} \varphi B\cdot\nabla_vf \le\frac{\lambda}{4}\int_{Q_R}\phi^2|\nabla_vf_{k}|^2+\frac{\Lambda^2}{\lambda}\int_{Q_R}\phi^2 f_{k}^2. 
	\end{split}
	\end{equation*}
Combining with the estimate for the source term
	\begin{equation*}
	\int_{Q_R}s\varphi\le
	\| f_k\|_{L^2(Q_R)} \| s\|_{L^{q}(Q_R)}|\{f_k>0\}\cap Q_R|^{\frac{1}{2}-\frac{1}{q}},
	\end{equation*}
	we conclude that
\begin{align}\label{degioigi-ener}
\|\nabla_v(\phi f_k)\|_{L^2(Q_R)}^2
\lesssim& \left(1+\|\partial_t\phi\|_{L^\infty}+\|\nabla_x\phi\|_{L^\infty}+\|\nabla_v\phi\|_{L^\infty}^2\right)
      \|f_k\|_{L^2({Q_R})}^2\nonumber\\
     &+ \| s\|_{L^{q}(Q_R)}^2|\{f_k>0\}\cap Q_R|^{1-\frac{2}{q}}\nonumber\\
\lesssim& \frac{1}{(R-r)^3}\|f_k\|_{L^2({Q_R})}^2+\| s\|_{L^{q}(Q_R)}^2|\{f_k>0\}\cap Q_R|^{1-\frac{2}{q}}. 
\end{align}

\medskip\noindent{\textbf{Step 2.}} The gain of integrability. \\
The technique involved in this step is almost the same as the ones used in the previous proof but here $\beta=1$. 
Taking a cut-off function $\eta$ valued in $[0,1]$, supported in $Q_{\overline{r}}$, $\eta|_{Q_{r}}\equiv 1$, and
\begin{equation*}
|\partial_t\eta|\le \frac{C}{(R-r)^2}, \quad
|\nabla_x\eta|\le \frac{C}{(R-r)^3}, \quad
|\nabla_v\eta|\le \frac{C}{R-r} {\quad\rm in\ } Q_1, 
\end{equation*}
we know that $\eta f_k$ is a subsolution to \eqref{whole}, where $g_1$, $g_0$ are replaced by 
\begin{equation*} 
\left\{ 
\begin{aligned}
\ &g_1=-Af_k\nabla_v\eta,  \\
\ &g_0=(B\eta-A\nabla_v\eta)\cdot\nabla_vf_k+f_k(\partial_t+b\cdot\nabla_x)\eta +s\mathbbm{1}_{f_k\ge 0}. 
\end{aligned}
\right. 
\end{equation*}
With the help of \eqref{degioigi-ener}, $g_1$ and $g_0$ can be estimated as follows: 
\begin{align*}
\| g_1\|^2_{L^2({Q_{\overline{r}}})}+\| g_0\|^2_{L^2({Q_{\overline{r}}})}
\lesssim& \left(\|\partial_t\eta\|_{L^\infty}^2+\|\nabla_x\eta\|^2_{L^\infty}+\|\nabla_v\eta\|_{L^\infty}^2\right) \| f_k\|^2_{L^2(Q_R)}\\ 
&+\left(1+\|\nabla_v\eta\|_{L^\infty}^2\right)\|\nabla_vf_k\|^2_{L^2(Q_{\overline{r}})} +\| s\|^2_{L^2(Q_R)}|\{f_k>0\}\cap Q_R|^{1-\frac{2}{q}}\\
\lesssim& \frac{1}{(R-r)^{6}}\| f_k\|^2_{L^2(Q_R)}
+\| s\|_{L^q(Q_R)}^2 |\{f_k>0\}\cap Q_R|^{1-\frac{2}{q}}. 
\end{align*}
By the same subsequent argument as in the previous proof, 
we can deduce that there is some barrier function $h\in H^\gamma(Q_r)$ dominating $f_k$ in $Q_r$ and satisfying 
\begin{equation*}\label{energydegiorgi0}
\| h\|_{H^\gamma(Q_r)}\lesssim \| g_1\|_{L^2({Q_{\overline{r}}})}+\| g_0\|_{L^2({Q_{\overline{r}}})}
\lesssim \frac{1}{(R-r)^3}\| f_{k}\|_{L^2(Q_R)}
+\| s\|_{L^q(Q_R)}|\{f_k>0\}\cap Q_R|^{\frac{1}{2}-\frac{1}{q}},
\end{equation*}
where $\gamma=\frac{2\alpha}{(4+\alpha)(4+d_2)}$. 
Then, recalling $0\le f_k\le h$ in $Q_r$ and the Sobolev inequality, we have
\begin{equation}\label{energydegiorgi}
\| f_{k}\|_{L^{2\kappa}(Q_{r})}\le \| h\|_{L^{2\kappa}(Q_{r})}
\lesssim \frac{1}{(R-r)^3}\| f_{k}\|_{L^2(Q_R)}
+\| s\|_{L^q(Q_R)}|\{f_k>0\}\cap Q_R|^{\frac{1}{2}-\frac{1}{q}},
\end{equation}
where $\kappa>1$ is given by \eqref{sobolevconjugate}. 

\medskip\noindent{\textbf{Step 3.}} Estimates of the superlevel sets. \\
By H\"{o}lder's inequality, 
	\begin{equation}\label{degio0}
	\int_{Q_{r}} f_k^2\le \left(\int_{Q_{r}} f_k^{2\kappa}\right)^{\frac{1}{\kappa}}|\{f_k>0\}\cap Q_R|^{1-\frac{1}{\kappa}}. 
	\end{equation}
Recall that \eqref{bddindex} is equivalent to $2-\frac{2}{q}-\frac{1}{\kappa}>1$. 
Then, \eqref{energydegiorgi} and \eqref{degio0} imply
\begin{align}\label{degio1}
	\int_{Q_r}f_k^2
	&\lesssim \frac{1}{(R-r)^6}|\{f_k>0\}\cap Q_R|^{1-\frac{1}{\kappa}} \int_{Q_R}f_k^2 
	+ |\{f_k>0\}\cap Q_R|^{2-\frac{2}{q}-\frac{1}{\kappa}}\| s\|_{L^q(Q_R)}^2\nonumber\\
	&\lesssim \frac{1}{(R-r)^6}|\{f_k>0\}\cap Q_R|^\epsilon \int_{Q_R}f_k^2 
	+ |\{f_k>0\}\cap Q_R|^{1+\epsilon}\| s\|_{L^q(Q_R)}^2,
\end{align}
	for some (small) universal constant $\epsilon>0$, provided that $|\{f_k>0\}\cap Q_R|\le1$. Indeed, due to Chebyshev's inequality, for any $j<k$, 
	\begin{equation}\label{degio2}
	|\{f_k>0\}\cap Q_R|=|\{f-j>k-j\}\cap Q_R|\le \frac{1}{(k-j)^2}\int_{Q_R}f_j^2.
	\end{equation}
	In particular, $|\{f_k>0\}\cap Q_R|\le1$ if $k\ge l_0:=C\| f^+\|_{L^2(Q_1)}$ for some (large) universal constant $C>0$. 
	Besides, for any $j<k$, we have $0\le f_k\le f_j$.
Combining this with \eqref{degio1} and \eqref{degio2}, we conclude that for any $k>j>l_0$,
\begin{equation*}
\|f_k\|_{L^2(Q_r)}
\lesssim \left(\frac{1}{(R-r)^3}+\frac{\| s\|_{L^q(Q_1)}}{k-j}\right)\frac{\|f_j\|_{L^2(Q_R)}^{1+\epsilon}}{(k-j)^\epsilon}.
\end{equation*}
	
\medskip\noindent{\textbf{Step 4.}} The recursion. \\
For some $l>0$, taking $k_n=l_0+l(1-\frac{1}{2^n})$, $r_n=\frac{1}{2}+\frac{1}{2^{n+1}}$ 
and recalling $\mathcal{A}_n=\|f_{k_n}\|_{L^2(Q_{r_n})}$, we derive 
	\begin{equation}\label{degiorgiinduction0}
\mathcal{A}_n \le \frac{C^n(l+\| s\|_{L^q(Q_1)})}{l^{1+\epsilon}}\mathcal{A}_{n-1}^{1+\epsilon}.
	\end{equation} 
It is not hard to show that there exists some constant $m>1$ such that for any $n\in\mathbb{N}$, 
	\begin{equation}\label{degiorgiinduction}
\mathcal{A}_n \le \frac{\mathcal{A}_0}{m^n}. 
	\end{equation} 
Indeed, \eqref{degiorgiinduction} holds for $n=0$ obviously. For any $n>1$, suppose that \eqref{degiorgiinduction} holds for $n-1$. Then, applying \eqref{degiorgiinduction0}, we have
\begin{equation*}
	\mathcal{A}_n
	\le\frac{C^n(l+\| s\|_{L^q(Q_1)})}{l^{1+\epsilon}}\left(\frac{\mathcal{A}_0}{m^{n-1}}\right)^{1+\epsilon}
	=\frac{C^n}{m^{\epsilon n-\epsilon-1}}\cdot\frac{l+\| s\|_{L^q(Q_1)}}{l}\cdot\frac{\mathcal{A}_0^\epsilon}{l^\epsilon}\cdot\frac{\mathcal{A}_0}{m^n},
\end{equation*} 
	which completes the induction by choosing $m^\epsilon=C$ and $l=C_0(\| s\|_{L^q(Q_1)}+\mathcal{A}_0)$ with some (large) universal constant $C_0>0$.
	
Finally, sending $n\rightarrow\infty$ in \eqref{degiorgiinduction}, we have 
$\mathcal{A}_\infty=\|f_{k_\infty}\|_{L^2(Q_{r_\infty})}=0$ with $k_\infty=l_0+l$ and $r_\infty=\frac{1}{2}$. 
Recalling the definition of $l_0=C\|f^+\|_{L^2(Q_1)}$, we get the desired result. 
\end{proof}

\begin{remark}
Let us consider the typical case that $d_1=d_2=d$, $\alpha=1$. 
If we additionally strengthen our assumption so that $\|b\|_{W^{1,\infty}}$ is bounded, then Remark~\ref{rkVAL} (with $\delta=1$) implies that $\gamma=\frac{1}{1+2d}$. 
From \eqref{bddindex} and \eqref{sobolevconjugate}, we then know that the range of the index $q$ is given by $q>(1+2d)^2$. 
We will use this range below as the assumptions above are all fulfilled in the next section. 
This makes a difference in the range of $q$ in Theorems~\ref{bdd} and \ref{regularity}. 
\end{remark}

\section{H\"older continuity}\label{sectionholder}
Passing from the local boundedness result in the previous section to the H\"older continuity of solutions to \eqref{FP} is based on De Giorgi's intermediate value lemma, which is a parabolic counterpart of De Giorgi's second lemma (see \cite{CV} and \cite{Va}, for instance). In the classical elliptic case, De Giorgi's originial lemma in \cite{DG}, relying on an ingredient of isoperimetric inequality, indicates that any function lying in $H^1$ should pay enough measure between two distinct values to jump from one to the other. 
This property does not hold anymore for $H^s$-function for any $s<\frac{1}{2}$, since any characteristic functions of bounded smooth domains lie there. 

In the hypoelliptic case or even parabolic case, such a result is neither valid for general $L^\infty_tH^1_v$ functions nor valid for subsolutions in any cylinder domains of definition, 
since subsolutions may drop in time right off, such as the function defined by taking identically one in the first half time and taking identically zero in the second half time. 
Instead, there should exist a time lag when the subsolutions pay in measure to be allowed to jump (see its precise statement in Lemma~\ref{lemmaDG} below). 
Furthermore, it is related to the propagation of positivity and zeros of solutions to \eqref{FP}; see \cite{GIMV}. The nondegeneracy of the term $b$ guarantees that the propagation is possible. As a matter of fact, how a nondegenerate operator propagates minima can be traced back to Bony’s maximum principle \cite{Bony}. 

Throughout this section, we assume $d_1=d_2=d$.

\subsection{Zooming in the equation}\label{zoom}
To establish the regularity result (Theorem~\ref{regularity}), we need to pay attention to the increment of solutions to \eqref{FP} in the infinitesimal neighborhood of each point. Consequently, a zooming procedure is necessary. 
Observing that the kinetic Fokker-Planck equation \eqref{FP} is not translation invariant in the velocity variable, we have to consider the transformation $\mathcal{T}_{z_0,r}$ by the prescript 
\begin{equation*}\label{zoomtrans}
\mathcal{T}_{z_0,r}:(\tilde{t},\tilde{x},\tilde{v})\longmapsto(t,x,v):=(t_0+r^2\tilde{t},x_0+r^3\tilde{x}+r^2\tilde{t}b(v_0),v_0+r\tilde{v}), 
\end{equation*}
which is the composition of scaling and translation. 
To zoom in on the equation, with $r\in(0,1)$, we define 
$$\tilde{f}:=f\circ\mathcal{T}_{z_0,r}.$$
Then, as soon as $f$ verifies \eqref{FP} in the domain 
\begin{equation}\label{domain}
Q^b_r(z_0):=\left\{(t,x,v):t\in(t_0-r^2,t_0],\ |x-x_0-(t-t_0)b(v_0)|<r^3,\ |v-v_0|<r\right\},
\end{equation}
$\tilde{f}$ is a solution to the equation in $Q_1$, 
\begin{equation}\label{FPzoom}
(\partial_{\tilde{t}}+b_r\cdot\nabla_{\tilde{x}})\tilde{f}={\rm div}_{\tilde{v}}\big(\tilde{A}\nabla_{\tilde{v}}\tilde{f}\big)+\tilde{B}\cdot\nabla_{\tilde{v}}\tilde{f}+\tilde{s},
\end{equation} 
where the new coefficients are defined by 
\begin{equation}\label{coefficient}
b_r(\tilde{v}):=\frac{1}{r}[b(v_0+r\tilde{v})-b(v_0)],
 \quad \tilde{A}:=A\circ\mathcal{T}_{z_0,r},\quad \tilde{B}:=rB\circ\mathcal{T}_{z_0,r}, 
 {\quad \rm and\quad } \tilde{s}:=r^2s\circ\mathcal{T}_{z_0,r}.
\end{equation}

Recalling the assumption in this section that $b\in C^1(B_1)$ and $\| Db\|_{L^\infty(B_1)}\le\Lambda$, we know that for any $B_{r}(v_0)\subset B_1$, 
\begin{equation*}
\| b_r\|_{W^{1,\infty}(B_1)}
\le 2\| Db\|_{L^\infty(B_r(v_0))}
\le 2\Lambda.
\end{equation*}
We point out that the new coefficients $\tilde{A},\tilde{B},$ and $b_r$ still satisfy the conditions \eqref{H} and \eqref{nond} with $\alpha=1$ in $Q_1$, as long as ${A},{B},$ and $b$ satisfy \eqref{H} and \eqref{N}. 
We will focus on this rescaled equation \eqref{FPzoom} from now on.

\subsection{Invertibility of $Db$}\label{invertibility}
Recall Hörmander's theorem \cite[Theorem~1.1]{Ho} for the hypoelliptic operators as mentioned in \S\ref{hypoelliptic}. If the system of vector fields is defined by 
\begin{equation*}
X_0:=\partial_t+b(v)\cdot\nabla_x 
{\quad \rm and\quad}
(X_1,X_2,\ldots ,X_{d})^T:=\sqrt{A}\nabla_v,
\end{equation*} 
 where the matrix $A(t,x,v)$ is uniformly positive-definite as in our assumption \eqref{H}, then
$$\sum_{i=1}^d X_i^*X_i+X_0=\partial_t+b\cdot\nabla_x-{\rm div}_v(A\nabla_v\cdot).$$
The vector fields $\{X_i\}_{i=0}^d$ together with their first order Lie brackets generate the full tangent space at each point of $Q_1$ if and only if the matrix $Db$ is invertible everywhere in $B_1$. This fact motivates us to extract some information on $Db$ from our nondegeneracy assumption \eqref{N}. 

We characterize the invertibility of $Db$ quantitatively thanks to the following lemma. 
\begin{lemma}\label{nondlemma}
If $b\in C^1(B_1,\mathbb{R}^d)$ is nondegenerate in the sense of \eqref{N}, then the spectral radius of the $d\times d$ matrix $(Db)^{-1}$ has some (positive) universal upper bound. 
Conversely, if the spectral radius of $(Db)^{-1}$ equals to $\kappa$ (positive) in $B_1$, then $b$ satisfies \eqref{N} in $B_\frac{1}{2}$ for some $K>0$ only depending on $d$, $\kappa$, and $\|Db\|_{L^\infty}$. 
\end{lemma}
\begin{proof}
In view of the condition \eqref{N} with the ball $B=B_r(v_0)$, as well as changing variables and selecting $\mu:=b(v_0)\cdot\nu$, we deduce
\begin{equation*}
\forall \nu\in\mathbb{S}^{d-1}, \ \forall \epsilon>0,\ \  
\bigg|\Big\{v\in B_1:\Big|\frac{1}{r}[b(v_0+rv)-b(v_0)]\cdot\nu\Big|\le \epsilon\Big\}\bigg| 
\le K\epsilon.
\end{equation*}
Since $b$ is differentiable in $B_1$, for any $v_0\in B_1$, by the dominated convergence theorem, we can send $r\rightarrow 0^+$ to see that
\begin{equation*}
\forall \nu\in\mathbb{S}^{d-1}, \ \forall \epsilon>0,\ \  
\big|\{v\in B_1:|\langle v,Db(v_0)^T\nu\rangle|\le \epsilon\}\big| 
\le K\epsilon.
\end{equation*}
Then, sending $\epsilon\rightarrow 0^+$ implies that for any $\nu\in\mathbb{S}^{d-1}$ and for a.e. $v\in B_1$, $\langle v,Db(v_0)^T\nu\rangle\neq0$, which shows that $Db(v_0)$ is invertible. 
Additionally, by symmetry of the ball, for a fixed  $\eta\in\mathbb{S}^{d-1}$ and for any $\epsilon>0$, we have 
\begin{equation*}
\big|\{v\in B_1: |\langle v,\eta\rangle|\le\| Db(v_0)^{-1}\|\epsilon\}\big|=
\bigg|\Big\{v\in B_1: \inf_{\nu\in\mathbb{S}^{d-1}} |Db(v_0)\nu|
|\langle v,\eta\rangle|\le \epsilon\Big\}\bigg| 
\le K\epsilon, 
\end{equation*}
which means 
\begin{equation*}
\| Db(v_0)^{-1}\|\le  C(d)K, 
\end{equation*}
where the constant $C(d)$ only depends on $d$.

Conversely, since $Db$ is invertible in $B_1$, based on a covering argument, it suffices to consider \eqref{N} in the ball $B_r(v_0)$  where the inverse mapping theorem can be applied on the $b$. 
Then, by changing variables, for any $\mu\in\mathbb{R}$, $\nu\in\mathbb{S}^{{d}-1}$, and $\epsilon>0$, 
\begin{align*}
\big|\{ v\in B_r(v_0): | \mu+b(v)\cdot\nu|\le \epsilon\}\big| 
&\le \sup\nolimits_{B_1}\det(Db^{-1})\int_{w+b(v_0)\in b\left(B_r(v_0)\right)} \mathbbm{1}_{| \mu+w \cdot\nu|\le \epsilon} \dif w \\
&\le C(d) \kappa^{d}\|Db\|_{L^\infty}^{d-1} r^{{d}-1}\epsilon. 
\end{align*}
This completes the proof. 
\end{proof}

\subsection{Intermediate value lemma}
Although changing variables is available on account of Lemma~\ref{nondlemma}, reformulating \eqref{FPzoom} in a way such that the transport part of the equation becomes a classical one $\partial_t+v\cdot\nabla_x$ is still impossible. 
Therefore, De Giorgi's intermediate value lemma for the free streaming case in \cite{GIMV} cannot be applied here directly. 

We are going to establish the counterpart of the intermediate value lemma by rephrasing the compactness argument presented in \cite[Lemma~4.1]{GIMV} additionally with a linearization technique. 
Roughly speaking, arguing by contradiction produces a sequence of subsolutions to the rescaled equation \eqref{FPzoom}. 
The contradiction will be followed in four steps. 
First, passing to a limit is possible due to certain compactness from the elliptic energy estimate and the averaging lemma. 
Second, the limit function will be independent of the velocity variable because of the elliptic energy estimate and the isoperimetric inequality of $H^1$-functions. 
Third, we are able to linearize the velocity vector field in the drift term of the limit inequation if the zoom scale $r^{-1}$ is large enough. 
Fourth, the propagation of the zeros of the limit function will lead to a contradiction. 

Before stating the lemma, we simplify some notation as follows. 
For some positive constants $\omega$ and $s_0$, we use the abbreviation $Q_\omega:=(-\omega^2,0]\times B_{\omega^3}\times B_\omega$ and $Q^-_\omega:=Q_\omega-(s_0,0,0)$. 
We will determine $\omega$ and $s_0$ to be universal and such that $Q_\omega$ and $Q^-_\omega$ do not overlap; indeed, the time lag $s_0-\omega^2$ between $Q_\omega$ and $Q^-_\omega$ is bounded from below by some universal positive constant. 

\begin{lemma}[De Giorgi's intermediate value lemma]\label{lemmaDG}
	Let $\delta_1,\delta_2\in (0,1)$ be universal and $f$ be a subsolution to \eqref{FPzoom} with $\tilde{A},\tilde{B}$
	satisfying \eqref{H} in $Q_1$. 
	Assume $b\in C^1(B_1)$, with $o_1$ as the modulus of continuity of $Db$, is nondegenerate in $B_1$ in the sense of \eqref{N}, and in addition, 
	$$\| Db\|_{L^\infty(B_1)}\le\Lambda {\quad and\quad} \|\tilde{s}\|_{L^q(Q_1)}\le\Lambda,$$ 
	for some $q>(1+2d)^2$.  
	There exist some (small) universal positive constants $\omega,s_0,r_0,\delta_0$, and $\theta$ such that 
	$Q_\omega\cup Q^-_\omega\subset Q_\frac{1}{2}$ and for any $r\in(0,r_0]$, if $f\le1$ in $Q_1$ satisfies
	\begin{equation*}
	\left\{ 
	\begin{aligned}
	\ &|\{f\ge1-\theta\}\cap Q_\omega|\ge\delta_1|Q_\omega|,\\
	\ &|\{f\le0\}\cap Q^-_\omega|\ge\delta_2|Q^-_\omega|,
	\end{aligned}
	\right. 
	\end{equation*} 
	then we have
	$$|\{0<f<1-\theta\}\cap Q_\frac{1}{2}|\ge\delta_0.$$
\end{lemma}

\begin{proof}
We will select the constant $r_0>0$ in Step~3 and the constants $\omega,s_0\in(0,1)$ such that $Q_\omega\cup Q^-_\omega\subset Q_\frac{1}{2}$ in Step~4 below. 
Arguing by contradiction, we suppose that there exist two sequences of positive constants 
\begin{equation*}
\{\theta_n\}_{n\in\mathbb{N}^+}
{\rm\ \ with\ }\theta_n\rightarrow 0,\quad 
\{r_n\}_{n\in\mathbb{N}^+}
{\rm\ \ with\ }r_n\rightarrow r\in[0,r_0],
\end{equation*}
and a sequence of subsolutions $\{f_n\}_{n\in\mathbb{N}^+}$ verifying
\begin{equation}\label{lemmaassumption0}
(\partial_t+b_{r_n}\cdot\nabla_x)f_n\le {\rm div}_v\left(A_n\nabla_v f_n\right)+{B}_n\cdot\nabla_vf_n+s_n {\quad\rm in\ }Q_1,
\end{equation}
with $f_n\le 1$ in $Q_1$, $A_n,B_n$ satisfying \eqref{H}, $\| s_n\|_{L^q(Q_1)}\le\Lambda$, and in addition,
	\begin{equation}\label{lemmaassumption}
	\left\{ 
	\begin{aligned}
	\ &|\{f_n\ge1-\theta_n\}\cap Q_\omega|\ge\delta_1|Q_\omega|,\\
	\ &|\{f_n\le0\}\cap Q^-_\omega|\ge\delta_2|Q^-_\omega|,\\
	\ &|\{0<f_n<1-\theta_n\}\cap Q_\frac{1}{2}|\rightarrow0 {\quad\rm as\ } n\rightarrow \infty. 
	\end{aligned}
	\right. 
	\end{equation}  

\medskip\noindent{\textbf{Step 1.}} Passage to the limit. \\
By Lemma~\ref{subsol}, $f_n^+$ is also a subsolution to \eqref{lemmaassumption0} with the source term replaced by $|s_n|$ so that there exists some nonnegative measure $\mu_n$ such that 
	\begin{equation}\label{lemmaeq}
	(\partial_t+b_{r_n}\cdot\nabla_x)f_n^+= {\rm div}_v\left(A_n\nabla_v f_n^+\right)+{B}_n\cdot\nabla_vf_n^++|s_n|-\mu_n.
	\end{equation}

First, $b_{r_n}\rightarrow b_r$ in $L^\infty(B_1)$, where $b_r(v)=Db(v_0)v$ if $r=0$.  
Then, take a cut-off function $\phi$ such that $0\le\phi\le 1$, $\phi\equiv 1$ in $Q_\frac{3}{4}$ and ${\rm supp}\phi\subset Q_1$. Since $f_n^+\le1$, integrating \eqref{lemmaeq} against $\phi^2 f_n^+$ as before, we are able to get the energy estimate
	\begin{equation}\label{lemmaenergy}
	\int_{Q_1}\phi^2|\nabla_vf_n^+|^2
	\lesssim \int_{Q_1}\left( 1+|\nabla_v\phi|^2+|\partial_t\phi|+|\nabla_x\phi|\right),
	\end{equation} 	
	so that, after passing to a subsequence, we have 
	$$f_n^+\overset{\ast}{\rightharpoonup}F {\ \rm in\ } L^\infty(Q_\frac{3}{4}),\quad \nabla_vf_n^+\rightharpoonup\nabla_vF {\ \rm in\ } L^2(Q_\frac{3}{4}),$$
	and there exists $G_1,G_0\in L^2(Q_\frac{3}{4})$ such that
	$$A_n\nabla_vf_n^+\rightharpoonup G_1,\quad {B}_n\cdot\nabla_vf_n^++|s_n|\rightharpoonup G_0 \quad{\rm in}\ L^2(Q_\frac{3}{4}).$$
	Furthermore, integrating \eqref{lemmaeq} against $\phi^2$ yields
	\begin{equation*}
	\begin{split}
	\int_{Q_1}\phi^2\dif \mu_n
	\lesssim& \int_{Q_1} \left(1+|\nabla_v\phi|^2+|\partial_t\phi|+|\nabla_x\phi|+\phi^2|\nabla_vf_n^+|^2\right)\\
	\lesssim& \int_{Q_1} \left(1+|\nabla_v\phi|^2+|\partial_t\phi|+|\nabla_x\phi|\right),
	\end{split}
	\end{equation*} 	
	so that, up to a subsequence,
	$$\mu_n\rightharpoonup\mu_\infty\quad{\rm in}\ \mathcal{M}(Q_\frac{3}{4}).$$
	
	Therefore, sending $n\rightarrow\infty$ in \eqref{lemmaeq}, we have
	\begin{equation}\label{lemmaeqF}
	(\partial_t+b_r\cdot\nabla_x)F= {\rm div}_vG_1+G_0-\mu_\infty 
	{\quad\rm in\ } Q_\frac{3}{4}.
	\end{equation} 
We also have $L^2$-strong convergence of $\{f^+_n\}$. Indeed, by the energy estimate \eqref{lemmaenergy}, 
\begin{align}\label{compactes}
\|f_n^+-F\|_{L^2(Q_\frac{1}{2})}
\le& \|f_n^+-f_n^+*_v\rho_\epsilon\|_{L^2(Q_\frac{1}{2})}+\|f_n^+*_v\rho_\epsilon-F*_v\rho_\epsilon\|_{L^2(Q_\frac{1}{2})}
  +\|F*_v\rho_\epsilon-F\|_{L^2(Q_\frac{1}{2})}\nonumber\\
\le& C\epsilon+\|f_n^+*_v\rho_\epsilon-F*_v\rho_\epsilon\|_{L^2(Q_\frac{1}{2})}, 
\end{align}
where $\{\rho_\epsilon(v)\}_{\epsilon\in(0,1]}$ is a mollifier sequence. 
For each fixed $\epsilon$, together with the energy estimate, we can apply velocity averaging \cite[Theorem~1.8, Remark~1.8]{BGP} to \eqref{lemmaeq} in $Q_\frac{3}{4}$, with $b_r\in L^\infty$ such that the nondegeneracy holds, to derive the compactness of $\{f^+_n*_v\rho_\epsilon\}_n$ in $L^2(Q_\frac{1}{2})$. 
Sending $n\rightarrow\infty$ and then $\epsilon\rightarrow0$ in \eqref{compactes}, we obtain 
\begin{equation*}
f_n^+\rightarrow F \quad{\rm in}\ L^2(Q_\frac{1}{2}). 
\end{equation*}
Then, with $Q_\omega\cup Q^-_\omega\subset Q_\frac{1}{2}$, passing to the limit in our assumption \eqref{lemmaassumption} gives
\begin{equation}\label{lemmaF}
\left\{ 
\begin{aligned}
	\ &|\{F=1\}\cap Q_\omega|\ge\delta_1|Q_\omega|,\\
	\ &|\{F=0\}\cap Q^-_\omega|\ge\delta_2|Q^-_\omega|,\\
	\ &|\{0<F<1\}\cap Q_\frac{1}{2}|=0.\\
\end{aligned}
\right. 
\end{equation} 
	
\medskip\noindent{\textbf{Step 2.}} Identification of the limit function. \\
We point out that \eqref{lemmaenergy} also implies the boundedness of $\nabla_vF$ in $L_{loc}^2(Q_1)$. 
Referring to the isoperimetric lemma, Lemma~\ref{append1} (see the appendix), we know that a characteristic function cannot lie in $H^1$ unless it is constant. Thus, \eqref{lemmaF} says that either $F(t,x,v)=0$ for a.e. $v\in B_\frac{1}{2}$ or $F(t,x,v)=1$ for a.e. $v\in B_\frac{1}{2}$, 
	which means 
	$$F(t,x,v)=\mathbbm{1}_P(t,x) {\quad\rm in\ }Q_\frac{1}{2},$$
	for some measurable set $P$ in $\widetilde{Q}_\frac{1}{2}=(-\frac{1}{4},0)\times B_\frac{1}{8}$. 
Moreover, due to \eqref{lemmaF}, the set $P$ satisfies 
	\begin{equation}\label{lemmaP}
	\left\{ 
	\begin{aligned}
	\ &|\widetilde{Q}_\omega\cap P|>0,\\
	\ &|\widetilde{Q}^-_\omega\big\backslash P|>0,\\
	\end{aligned}
	\right. 
	\end{equation}
	where we denote $\widetilde{Q}_\omega:=(-\omega^2,0]\times B_{\omega^3}$ and $\widetilde{Q}^-_\omega:=\widetilde{Q}_\omega-(s_0,0)$.

\medskip\noindent{\textbf{Step 3.}} Linearization of the drift. \\
We are going to determine the zoom scale $r_0^{-1}$ to be sufficiently large so that $b_r$ with $r\in[0,r_0]$ behaves like a nondegenerate affine transformation. 

For any fixed $w\in B_\frac{1}{4}$, integrating \eqref{lemmaeqF} against $\rho(v-w)$ over $v\in B_1$, where the function $\rho\ge 0$ is radial, supported in $B_\frac{1}{4}$ and $\int_{B_1}\rho=1$, we deduce that, for any $(t,x,w)\in \widetilde{Q}_{\frac{1}{2}}\times B_\frac{1}{4}$, 
\begin{equation}\label{lemmaeqtx0}
(\partial_t+\mathscr{B}(w)\cdot\nabla_x)\mathbbm{1}_P(t,x)
\le  \|\rho\|_{W^{1,\infty}}\int_{B_1}\left(|G_1(t,x,v)|+|G_0(t,x,v)|\right)\dif v,
\end{equation} 
where the velocity vector field in the drift is defined by 
\begin{equation*}
\mathscr{B}(w):=\int_{B_1}b_r(v)\rho(v-w)\dif v.
\end{equation*}
On account of the definition of $b_r$ and our choice of the function $\rho$, 
\begin{equation}\label{bv}
\mathscr{B}(w)=\int_{B_1}\left(Db(v_0)v+o(r,v)\right)\rho(v-w)\dif v=Db(v_0)w+\int_{B_1}o(r,v)\rho(v-w)\dif v, 
\end{equation} 
where $o(r,v)$ is the remainder of the linearization of the velocity vector field $\mathscr{B}(w)$, 
\begin{equation}\label{bvo}
o(r,v):=b_r(v)-Db(v_0)v=\int_0^1 \left[Db(v_0+\tau rv)-Db(v_0)\right]v\dif\tau. 
\end{equation} 

By Lemma~\ref{nondlemma}, the set $\big\{Db(v_0)w:w\in B_\frac{1}{4}\big\}$ contains the ball $B_{2\sigma}$ for some universal  constant $\sigma>0$. 
The expression \eqref{bvo} gives $|o(r,v)|\le o_1(r)$ for any $v\in B_1$. 
Combining these two facts with \eqref{bv}, we conclude that there exists some universal positive constant $r_0$ (depending on $o_1$) such that $\mathscr{B}(w)$ takes all the values over the ball $B_\sigma$ 
whenever $r\in[0,r_0]$. 

 \begin{figure}
	\def\svgwidth{10cm}
	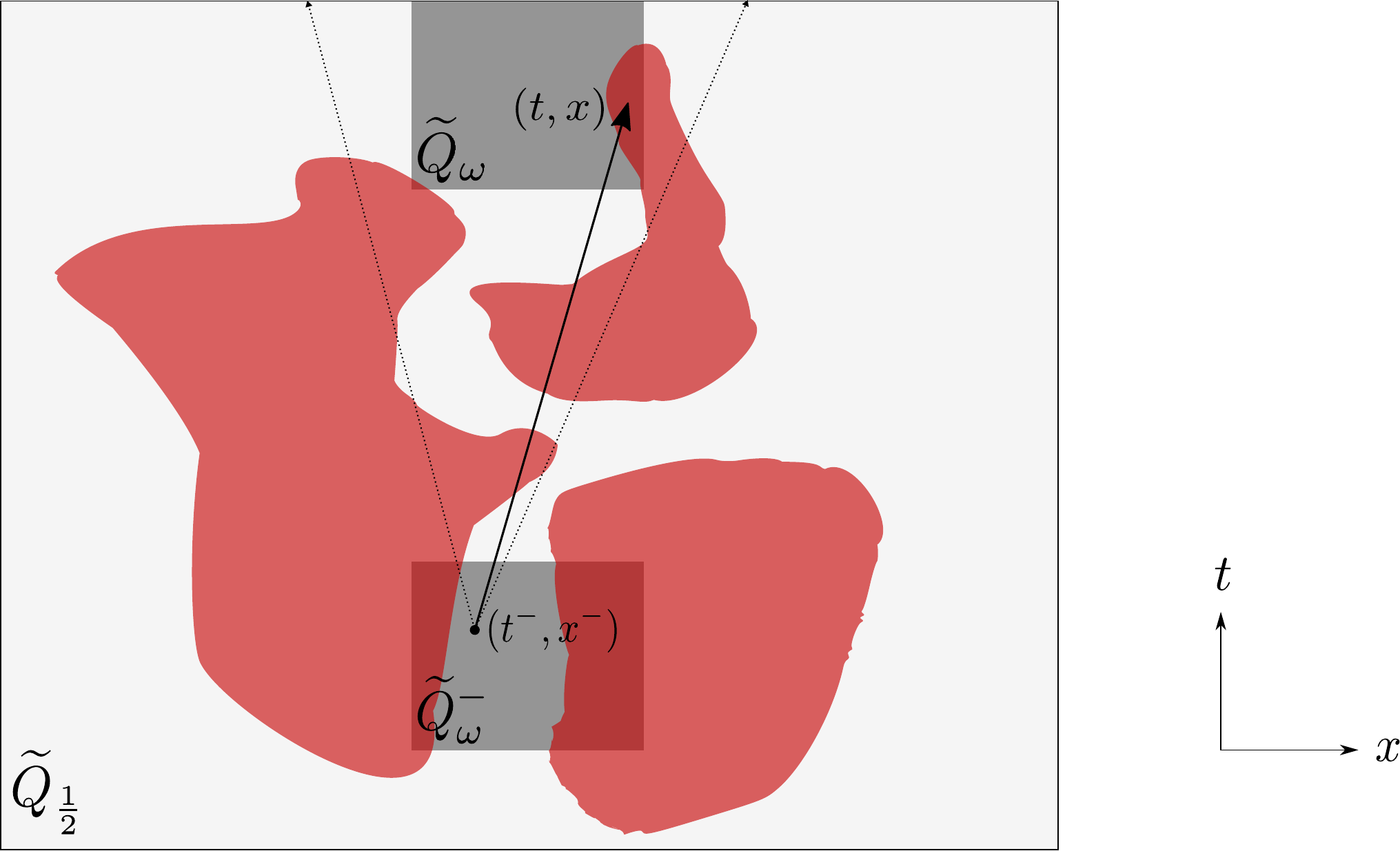\\\vspace{0.25cm}
	{\textbf{Figure.}} Any point in $\widetilde{Q}^-_\omega$ traveling with velocity valued in $B_\sigma$ can
	reach anywhere in $\widetilde{Q}_\omega$. The red region $P$ meets $\widetilde{Q}_\omega$ and does not cover $\widetilde{Q}^-_\omega$. 
\end{figure}

\medskip\noindent{\textbf{Step 4.}} Propagation of zeros. \\
Observe that the right-hand side of \eqref{lemmaeqtx0} is independent of $w$ and belongs to $L^2_{t,x}$. By applying Lemma~\ref{append1}, we can enhance the inequation \eqref{lemmaeqtx0} as follows: 
	\begin{equation}\label{lemmaeqtx}
	(\partial_t+\overline{v}\cdot\nabla_x)\mathbbm{1}_P(t^-+s,x^-+\overline{v}s)\le 0
	\end{equation} 
for any $(t^-,x^-,\overline{v})\in \widetilde{Q}^-_\omega\times B_\sigma$ and for any $s>0$ such that $(t^-+s,x^-+\overline{v}s)\in \widetilde{Q}_{\frac{1}{2}}$.
Then, for any $(t,x)\in\widetilde{Q}_\omega$, 
we pick $s$ and $\overline{v}$ verifying $(t,x)=(t^-+s,x^-+\overline{v}s)$. It follows that the time shift $s=t-t^-\in(s_0-\omega^2,s_0+\omega^2)$ and 
\begin{equation}\label{lemmaeqtx1}
|\overline{v}|=\frac{|x-x^-|}{s}<\frac{2\omega^2}{s_0-\omega^2} \le\sigma,
\end{equation} 
where we can see that $s_0>\omega^2$ and the last inequality in \eqref{lemmaeqtx1} both hold if we set 
$\omega:=\sqrt{\frac{\sigma s_0}{2+\sigma}}.$
Besides, by choosing $s_0:=\frac{1}{8}$, 
we have $s_0+\omega^2<\frac{1}{4}$ and $\omega^3<\frac{1}{8}$, which guarantees that 
$$Q_\omega\cup Q^-_\omega\subset Q_\frac{1}{2}.$$
In other words, we conclude that, for any $(t^-,x^-,t,x)\in \widetilde{Q}^-_\omega\times\widetilde{Q}_\omega$, there exists $\overline{v}\in B_\sigma$ such that $(t,x)=(t^-+s,x^-+\overline{v}s)$; see Figure. 
Consequently, from \eqref{lemmaeqtx}, we reach
\begin{equation*}
\mathbbm{1}_P(t,x)\le\mathbbm{1}_P(t^-,x^-) {\rm\quad for\ a.e.\ }(t^-,x^-,t,x)\in \widetilde{Q}^-_\omega\times\widetilde{Q}_\omega. 
\end{equation*}
Owing to the second inequality in \eqref{lemmaP}, we can choose $(t^-,x^-)\in \widetilde{Q}^-_\omega\backslash P$ so that 
\begin{equation*}
\mathbbm{1}_P(t,x)= 0 {\rm\quad for\ a.e.\ }(t,x)\in \widetilde{Q}_\omega,
\end{equation*}
which contradicts the first inequality in \eqref{lemmaP}. This finishes the proof. 
\end{proof}

\subsection{Density estimate and oscillation lemma}
\begin{lemma}[density estimate]\label{density}
Let $\delta_2>0$ be universal and $f$ be a subsolution to \eqref{FPzoom} with $\tilde{A},\tilde{B}$ satisfying \eqref{H} in $Q_1$. 
Assume $b\in C^1(B_1)$, with $o_1$ as the modulus of continuity of $Db$, is nondegenerate in $B_1$ in the sense of \eqref{N} and $\| Db\|_{L^\infty(B_1)}\le\Lambda$.
There exist some universal constants $\omega,s_0,r_0,\lambda_0>0$, and $\theta_0\in(0,1)$ 
such that for any $r\in(0,r_0]$, if $f\le1$ in $Q_1$ satisfies 
$$\left|\{f\le 0\}\cap Q^-_\omega\right|\ge \delta_2|Q^-_\omega| {\quad and\quad} \|\tilde{s}\|_{L^q(Q_1)}\le\lambda_0$$ 
for some $q>(1+2d)^2$, then
\begin{equation*}
	 f\le 1-\theta_0 {\quad in\ }Q_\frac{\omega}{2}. 
\end{equation*} 
\end{lemma}

The above density estimate shows that if a subsolution is far away from its upper bound in a subset of $Q^-_\omega$ with nontrivial measure, then it cannot get close to this bound in $Q_\frac{\omega}{2}$. 
Its proof is based on De Giorgi's intermediate value lemma, Lemma~\ref{lemmaDG}, and the local boundedness theorem, Theorem~\ref{bdd}, whose argument is almost the same as \cite[Lemma~4.5]{GIMV} and thus we omit it here.

\begin{lemma}[oscillation lemma]\label{osc}
Let $f$ be a weak solution to \eqref{FPzoom} with $\tilde{A},\tilde{B}$ satisfying \eqref{H} in $Q_1$. 
Assume $b\in C^1(B_1)$, with $o_1$ as the modulus of continuity of $Db$, is nondegenerate in $B_1$ in the sense of \eqref{N}, and in addition, 
$$\| Db\|_{L^\infty(B_1)}\le\Lambda {\quad and\quad} \tilde{s}\in L^q(Q_1)$$ 
for some $q>(1+2d)^2$.  
Then, there exist some universal constants $\omega,r_0,\theta_1\in(0,1)$ and $C>0$ such that for any $r\in(0,r_0]$, we have 
\begin{equation*}
{\rm osc}_{Q_\frac{\omega}{2}} f \le \left(1-\theta_1\right){\rm osc}_{Q_1} f+C\|\tilde{s}\|_{L^q(Q_1)}. 
\end{equation*} 
\end{lemma}

\begin{proof}
We have shown the local boundedness of the weak solutions. Thus, for the sake of clarity, we assume $f$ is bounded from above and from below in $Q_1$ so that we can define 
\begin{equation*}\label{rescale}
\overline{f}(t,x,v):=\left({\rm osc}_{Q_1}f+2\lambda_0^{-1}\|\tilde{s}\|_{L^q(Q_1)}\right)^{-1} \left(\sup\nolimits_{Q_1}f+\inf\nolimits_{Q_1}f-2f(t,x,v)\right), 
\end{equation*}
where the constant $\lambda_0>0$ is provided by Lemma~\ref{density}. 

Observing that $-1\le f\le1$ and $\pm\overline{f}$ are weak solutions to the equation
\begin{equation*}
(\partial_{t}+b_r\cdot\nabla_{x})(\pm\overline{f})={\rm div}_{v}\big(\tilde{A}\nabla_{v}(\pm\overline{f})\big)+\tilde{B}\cdot\nabla_{v}(\pm\overline{f})\mp\overline{s} {\quad\rm in\ }Q_1, 
\end{equation*} 
where $\overline{s}:=\left(\frac{1}{2}{\rm osc}_{Q_1}f+\lambda_0^{-1}\|\tilde{s}\|_{L^q(Q_1)}\right)^{-1}\tilde{s}$,  verifies the condition $\|\overline{s}\|_{L^q(Q_1)}\le\lambda_0$. 
Then, we are allowed to apply Lemma~\ref{density} to $\pm \overline{f}$ in $Q_1$ with $\delta_2=\frac{1}{2}$  
and this determines the universal constants $\omega$ and $r_0$. 

If $\overline{f}$ is at least half of the space less than $0$ in $Q^-_\omega$, that is, $\big|\{\overline{f}\le 0\}\cap Q^-_\omega\big|\ge\frac{1}{2}\big|Q^-_\omega\big|$, then we have $\sup_{Q_\frac{\omega}{2}}\overline{f}\le 1-\theta_0$ so that
\begin{equation*}
{\rm osc}_{Q_\frac{\omega}{2}}f
=\frac{1}{2}\left({\rm osc}_{Q_1}f+2\lambda_0^{-1}\|\tilde{s}\|_{L^q(Q_1)}\right){\rm osc}_{Q_\frac{\omega}{2}}\overline{f}
\le\left(1-\frac{\theta_0}{2}\right)\left({\rm osc}_{Q_1}f+2\lambda_0^{-1}\|\tilde{s}\|_{L^q(Q_1)}\right). 
\end{equation*}
If $\overline{f}$ is at least half of the space greater than $0$ in $Q^-_\omega$, we can get the same result by dealing with $-\overline{f}$. This completes the proof with $\theta_1:=\frac{\theta_0}{2}$. 
\end{proof}

The H\"older estimate of the weak solution to \eqref{FP} is obtained by translation, scaling and applying Lemma~\ref{osc} iteratively. 
For convenience, we introduce a technical lemma \cite[Lemma~8.23]{GT} to take the place of the iterative procedure. 
\begin{lemma}\label{technicallemma2}
	Let $r_1>0$ and $\psi_1$, $\psi_2$ be nondecreasing functions on $[0,r_1]$ satisfying 
	\begin{equation*}
	\psi_1(\tau r)\le \delta \psi_1(r)+\psi_2(r)
	\end{equation*}
	for some $\tau,\delta\in (0,1)$. Then, for any $\epsilon\in(0,1)$ and for $r\in(0,r_1]$, there exist some positive constants $C=C(\delta)$ and $\beta_1:=\frac{(1-\epsilon)\log \delta}{\log \tau}$ such that 
	\begin{equation*}
	\psi_1(r)\le C\left(\left(\frac{r}{r_1}\right)^{\beta_1}\psi_1(r_1)+\psi_2\left(r^\epsilon r_1^{1-\epsilon}\right) \right).
	\end{equation*}
\end{lemma}

We are now ready to establish the H\"older regularity of solutions. 
\begin{proof}[Proof of Theorem~\ref{regularity}]
Let us first write down the rescaled equation to apply Lemma~\ref{osc}. 
Let $z_0:=(t_0,x_0,v_0)\in Q_1$ and $r_1:=\min\big\{r_0,\sup\{r>0:Q^b_{r}(z_0)\subset Q_1\}\big\}$ with $r_0$ given in Lemma~\ref{osc}.  
With $r\in(0,r_1]$, as in subsection~\ref{zoom}, $\tilde{f}:=f\circ\mathcal{T}_{z_0,r}$ satisfies \eqref{FPzoom} with the new coefficients defined by \eqref{coefficient} which satisfy the conditions \eqref{H} and \eqref{N} in $Q_1$.

By Lemma~\ref{osc} applied to $\tilde{f}$, we obtain
\begin{equation*}
{\rm osc}_{Q_\frac{\omega}{2}}\tilde{f} \le \left(1-\theta_1\right){\rm osc}_{Q_1}\tilde{f}+C\| \tilde{s}\|_{L^q(Q_1)}. 
\end{equation*} 
Rescaling back, we deduce that, for any $r\in[0,r_1]$, 
\begin{equation*}
{\rm osc}_{Q^b_\frac{\omega r}{2}(z_0)}f \le \left(1-\theta_1\right){\rm osc}_{Q^b_r(z_0)}f+Cr^{2-\frac{1+2d}{q}}\| s\|_{L^q(Q^b_r(z_0))}. 
\end{equation*} 

Then, applying Lemma~\ref{technicallemma2} with $\tau:=\frac{\omega}{2}$, $\delta:=1-\theta_1$,
$$\psi_1(r):={\rm osc}_{Q^b_r(z_0)}f,\quad \psi_2(r):=Cr^{2-\frac{1+2d}{q}}\| s\|_{L^q(Q^b_r(z_0))}, $$
and choosing 
$\epsilon\in(0,1)$ such that $$\beta_1=\frac{(1-\epsilon)\log\delta}{\log\tau}\le\left(2-\frac{1+2d}{q}\right)\epsilon$$
yields that, for any $r\in[0,r_1]$, 
\begin{align}\label{osceq}
{\rm osc}_{Q^b_r(z_0)}f
\lesssim& \left(\frac{r}{r_1}\right)^{\beta_1}{\rm osc}_{Q^b_{r_1}(z_0)}f+\left(r^\epsilon r_1^{1-\epsilon}\right)^{2-\frac{1+2d}{q}}\| s\|_{L^q(Q^b_{r_1}(z_0))}\nonumber\\
\lesssim& \left(\frac{r}{r_1}\right)^{\beta_1}\left({\rm osc}_{Q^b_{r_1}(z_0)}f+r_1^{2-\frac{1+2d}{q}}\| s\|_{L^q(Q^b_{r_1}(z_0))}\right). 
\end{align}
Recalling the definition \eqref{domain} for $Q^b_r(z_0)$ with $r\le1$, we point out that $(t_0-r^3,t_0]\times B_{cr^3}(x_0)\times B_{r^3}(v_0)\subset Q^b_r(z_0)$ for some (small) positive constant $c$ only depending on $\Lambda$. 
Combining this with \eqref{osceq} and Theorem~\ref{bdd}, as well as using a standard covering argument, we achieve 
\begin{equation*}
|f(z)-f(z')|
\lesssim |z-z'|^\frac{\beta_1}{3}\left(\| f\|_{L^2(Q_1)}+\| s\|_{L^q(Q_1)}\right)
\end{equation*}
for any $z,z'\in Q_{\frac{1}{2}}$. This finishes the proof with $\beta:=\frac{\beta_1}{3}$. 
\end{proof}

\subsection{Boundary estimate}
When dealing with nonlinear Cauchy problems, one will always come across the evolution characterization of solutions from some given initial data. In other words, apart from interior estimates, a priori estimates around the initial time have to be derived. Based on our previous work, we state such kind of result precisely as follows. 

\begin{corollary}
Let $f$ be a weak solution to \eqref{FP} in $[0,1]\times B_1\times B_1$. 
In addition to the assumptions in Theorem~\ref{regularity} (replace $Q_1$ by $[0,1]\times B_1\times B_1$), we suppose that the initial data $f(0,x,v)\in C^{\alpha_0}(B_1\times B_1)$. 
Then, there exist some constants $\beta_0$, $C>0$ depending only on $\lambda,\Lambda,d,K,q,o_1,$ and $\alpha_0$
such that 
	\begin{equation*}
	\| f \|_{C^{\beta_0}\big([0,1]\times B_\frac{1}{2}\times B_\frac{1}{2}\big)}
	\le C\left(\| f\|_{L^2([0,1]\times B_1\times B_1)}+\| s\|_{L^q([0,1]\times B_1\times B_1)}
	+\| f(0,\cdot,\cdot)\|_{C^{\alpha_0}(B_1\times B_1)} \right).  
	\end{equation*}
\end{corollary}
\begin{proof}
It suffices to show the estimate around the initial time. For $z_0:=(t_0,x_0,v_0)$ with $t_0=0$,  $(x_0,v_0)\in B_1\times B_1$, we set the domains $\mathcal{C}_R:=(-R^2,R^2)\times B_{R^3}\times B_R$ for $R>0$ and
$$\mathcal{C}^b_r(z_0):=\left\{(t,x,v):t\in(t_0-r^2,t_0+r^2),\ |x-x_0-(t-t_0)b(v_0)|<r^3,\ |v-v_0|<r\right\},$$
where $r\in(0,r_1]$ and $r_1:=\min\big\{r_0,\sup\{r>0:\mathcal{C}^b_{r}(z_0)\subset \mathcal{C}_1\}\big\}$ with $r_0$ given by Lemma~\ref{density} with $\delta_2=\frac{1}{2}$. 

Let $\tilde{f}:=f\circ\mathcal{T}_{z_0,r}$ and $M:=\sup_{B_1\times B_1}\tilde{f}(0,\cdot,\cdot)$. Referring to Remark~\ref{subsol}, it is not hard to see that $\underline{f}:=(\tilde{f}-M)^+$ is a subsolution to \eqref{FPzoom} in $\mathcal{C}_1$ with the source term replaced by $|\tilde{s}|$, as it has zero extension for $t<0$. 

It follows that the local boundedness (Theorem~\ref{bdd}) around the initial time holds so that 
\begin{align}\label{bdbd}
\sup\nolimits_{\mathcal{C}_\frac{1}{2}}\tilde{f}
&\lesssim \|\underline{f}\|_{L^2({\mathcal{C}_1})} +\|\tilde{s}\|_{L^q(\mathcal{C}_1)}+M\nonumber\\
&\lesssim \| f\|_{L^2({[0,1]\times B_1\times B_1})}
	+\| s\|_{L^q([0,1]\times B_1\times B_1)}+\| f(0,\cdot,\cdot)\|_{L^\infty({B_1\times B_1})}.
\end{align}

Since $\underline{f}=0$ in $\mathcal{C}_1\cap\{t\le0\}$, 
applying Lemma~\ref{density} to
$\left(\sup\nolimits_{\mathcal{C}_1}\underline{f} +\lambda_0^{-1}\|\tilde{s}\|_{L^q(\mathcal{C}_1)}\right)^{-1} \underline{f}(t,x,v)$ with $\delta_2=\frac{1}{2}$ yields
\begin{align*}
\tilde{f}-M\le \underline{f}
\le (1-\theta_0)\sup\nolimits_{\mathcal{C}_1}\underline{f} +C\|\tilde{s}\|_{L^q(\mathcal{C}_1)}
= (1-\theta_0)\left(\sup\nolimits_{\mathcal{C}_1}\tilde{f}-M\right) +C\|\tilde{s}\|_{L^q(\mathcal{C}_1)} {\quad \rm in\ } \mathcal{C}_\omega
\end{align*}
for some constants $\omega,\theta_0\in(0,1)$. 
Similarly, $(m-\tilde{f})^+$ with $m:=\inf_{B_1\times B_1}\tilde{f}(0,\cdot,\cdot)$ is also a subsolution in $\mathcal{C}_1$ and such that $(m-\tilde{f})^+= 0$ in $\mathcal{C}_1\cap\{t\le0\}$, thus we deduce that
\begin{align*}
m-\tilde{f}\le -(1-\theta_0)\left(\inf\nolimits_{\mathcal{C}_1}\tilde{f}-m\right)+C\|\tilde{s}\|_{L^q(\mathcal{C}_1)}
{\quad\rm in\ } \mathcal{C}_\omega. 
\end{align*}
Adding them together, we have
\begin{align*}
	{\rm osc}_{\mathcal{C}_\omega}\tilde{f}
&\le (1-\theta_0){\rm osc}_{\mathcal{C}_1}\tilde{f}+\theta_0(M-m)+C\|\tilde{s}\|_{L^q(\mathcal{C}_1)}\\
&\le (1-\theta_0){\rm osc}_{\mathcal{C}_1}\tilde{f}+{\rm osc}_{B_1\times B_1}\tilde{f}(0,\cdot,\cdot)
  +Cr^{2-\frac{1+2d}{q}}\| s\|_{L^q([0,1]\times B_1\times B_1)}.
\end{align*}

Notice that $B_{cr^3}(z_0)\subset\mathcal{C}^b_r(z_0)$ for some universal constant $c>0$. Thanks to Lemma~\ref{technicallemma2}, there exist some universal constants $C$ and ${\beta_1}>0$ such that for any $r\in(0,r_1]$, 
	\begin{equation*}
	\begin{split}
{\rm osc}_{B_{cr^3}(z_0)}f &\le	{\rm osc}_{\mathcal{C}^b_r(z_0)}f
	\lesssim r^{\beta_1}\|\tilde{f}\|_{L^\infty(\mathcal{C}_\frac{1}{2})}
	  +r^{\beta_1}\| s\|_{L^q([0,1]\times B_1\times B_1)}
	  +{\rm osc}_{\mathcal{C}^b_{\sqrt{r}}(z_0)}f\\
	&\lesssim r^{\beta_1}\|\tilde{f}\|_{L^\infty(\mathcal{C}_\frac{1}{2})}
	  +r^{\beta_1}\| s\|_{L^q([0,1]\times B_1\times B_1)}
	  +r^\frac{\alpha_0}{2}\left[f(0,\cdot,\cdot)\right]_{C^{\alpha_0}(B_1\times B_1)}\\
	&\lesssim r^{3\beta_0}\left(\| f\|_{L^2({[0,1]\times B_1\times B_1})}
	  +\| s\|_{L^q([0,1]\times B_1\times B_1)}
	  +\| f(0,\cdot,\cdot)\|_{C^{\alpha_0}(B_1\times B_1)}\right), 
	\end{split}
	\end{equation*}
where we used \eqref{bdbd} and $\beta_0:=\frac{1}{3}\min(\frac{\alpha_0}{2},{\beta_1})$ in the last inequality. 
The proof is complete. 
\end{proof}

\appendix
\section{A qualitative De Giorgi's isoperimetric lemma}
The following result is intuitive. The derivative of characteristic functions behaves like delta functions and thus its boundedness from above by $L^p$-functions ($p>1$) implies its negative. 
This can be seemed as a qualitative version of De Giorgi's isoperimetric inequality (see, for instance, \cite[Lemma~II]{DG} and \cite[Lemma~10]{Va}). 
We give a proof for the sake of completeness. 

\begin{lemma}\label{append1}
Let $D\in\mathbb{N}^+$, $h\in\mathbb{S}^{D-1}$, $\Omega$ be an open set in $\mathbb{R}^{D}$ and $P$ be a measurable set in $\Omega$. 
If $h\cdot\nabla\mathbbm{1}_P\le g$ in $\mathcal{D}'(\Omega)$
for some $g\in L^p(\Omega)$ with $p>1$, then we actually have $h\cdot\nabla\mathbbm{1}_P\le 0$  in $\mathcal{D}'(\Omega)$.  
\end{lemma}

\begin{proof}
For any $\varphi\in\mathcal{D}(\Omega)$ with $\varphi\ge0$ and for any $\tau\in\left(0,{\rm dist}({\rm supp}\varphi,\partial\Omega)\right)$, 
\begin{equation*}
\begin{split}
\int_\Omega \frac{1}{\tau}\left(\varphi(z-\tau h)-\varphi(z)\right)\mathbbm{1}_P(z)\dif z
&= \int_\Omega \frac{1}{\tau}\left(\mathbbm{1}_P(z+\tau h)-\mathbbm{1}_P(z)\right)\varphi(z)\dif z\\
&\le \int_\Omega \frac{1}{\tau}\left(\mathbbm{1}_P(z+\tau h)-\mathbbm{1}_P(z)\right)_+^p\varphi(z)\dif z\\
&\le \int_\Omega \frac{1}{\tau}\left(\int_0^1 \tau h\cdot\nabla\mathbbm{1}_P(z+s\tau h)\dif s\right)_+^p\varphi(z)\dif z.
\end{split}
\end{equation*}
By our assumption, the last term above is bounded by $\tau^{p-1} \|g\|_{L^p(\Omega)}^{p}\|\varphi\|_{L^\infty(\Omega)}$. 
Then, applying the dominated convergence theorem with $\tau\rightarrow0^+$, we obtain
\begin{equation*}
-\int_\Omega \left(h\cdot\nabla\varphi(z)\right)\mathbbm{1}_P(z)\dif z\le 0. 
\end{equation*}
This concludes the proof. 
\end{proof}

\bibliographystyle{plain}
\bibliography{ZHU_FP}

\end{document}

%% file: FP-Figue.pdf_tex
\begingroup%
  \makeatletter%
  \providecommand\color[2][]{%
    \errmessage{(Inkscape) Color is used for the text in Inkscape, but the package 'color.sty' is not loaded}%
    \renewcommand\color[2][]{}%
  }%
  \providecommand\transparent[1]{%
    \errmessage{(Inkscape) Transparency is used (non-zero) for the text in Inkscape, but the package 'transparent.sty' is not loaded}%
    \renewcommand\transparent[1]{}%
  }%
  \providecommand\rotatebox[2]{#2}%
  \newcommand*\fsize{\dimexpr\f@size pt\relax}%
  \newcommand*\lineheight[1]{\fontsize{\fsize}{#1\fsize}\selectfont}%
  \ifx\svgwidth\undefined%
    \setlength{\unitlength}{584.7691412bp}%
    \ifx\svgscale\undefined%
      \relax%
    \else%
      \setlength{\unitlength}{\unitlength * \real{\svgscale}}%
    \fi%
  \else%
    \setlength{\unitlength}{\svgwidth}%
  \fi%
  \global\let\svgwidth\undefined%
  \global\let\svgscale\undefined%
  \makeatother%
  \begin{picture}(1,0.60777679)%
    \lineheight{1}%
    \setlength\tabcolsep{0pt}%
    \put(0,0){\includegraphics[width=\unitlength,page=1]{FP-Figue.pdf}}%
  \end{picture}%
\endgroup%